\documentclass[10pt]{amsart}
\usepackage{geometry}
\usepackage[dvips]{graphicx}
\usepackage{hyperref}
\usepackage[english]{babel}
\usepackage[T1]{fontenc}
\usepackage{textcomp}

\usepackage{amscd,amsfonts,amssymb,amsmath,amsthm,latexsym}

\usepackage{color}

\usepackage{graphics}
\usepackage{tikz}
\usepackage{tikz-cd} 
\usetikzlibrary{decorations.markings}
\usepackage[all]{xy}
\xyoption{curve}
\xyoption{import}
\xyoption{arc}
\xyoption{ps}

\theoremstyle{plain}
    \newtheorem{thm}{Theorem}[section]
    \newtheorem{lemma}[thm]{Lemma}
    \newtheorem{coro}[thm]{Corollary}
    \newtheorem{prop}[thm]{Proposition}

    \newtheorem*{Thm}{Theorem}
\theoremstyle{definition}
    \newtheorem{defi}[thm]{Definition}
    \newtheorem{defilem}[thm]{Definition/Lemma}
    \newtheorem{ex}[thm]{Example}
\theoremstyle{remark}
    \newtheorem{remark}{Remark}

\newcommand{\suchthat}{\ | \ }

\newcommand{\RQ}[1]{\mathbb{C}\langle\hspace{-0.75mm}\langle #1\rangle\hspace{-0.75mm}\rangle}
\newcommand{\myid}{1\hspace{-0.75mm}1}

\newcommand{\idealM}{\mathfrak{m}}

\newcommand{\jacobalg}[1]{\Lambda(#1)}

\newcommand{\rad}{\operatorname{rad}}
\newcommand{\coker}{\operatorname{coker}}
\newcommand{\image}{\operatorname{im}}

\newcommand{\bbZ}{\mathbb{Z}}
\newcommand{\bbC}{\mathbb{C}}

\newcommand{\Hom}{\operatorname{Hom}}

\newcommand{\ii}{{\bf i}}
\newcommand{\ee}{{\bf e}}
\newcommand{\Gr}{Gr}
\newcommand\restr[2]{{
  \left.\kern-\nulldelimiterspace 
  #1 
  \vphantom{\big|} 
  \right|_{#2} 
  }}

\begin{document}

\title[Landau-Ginzburg potentials via projective representations]{Landau-Ginzburg potentials via projective representations}
\author{Daniel Labardini-Fragoso}
\address{Daniel Labardini-Fragoso\newline
Instituto de Matem\'aticas, UNAM, Mexico, and\newline
Mathematisches Institut, Universit\"at zu K\"oln, Germany}
\email{labardini@im.unam.mx, dlabardi@uni-koeln.de}
\author{Bea de Laporte}
\address{Bea de Laporte\newline
Mathematisches Institut, Universit\"at zu K\"oln, Germany}
\email{bschumann@math.uni-koeln.de}
	\begin{abstract} We interpret the Landau-Ginzburg potentials associated to Gross-Hacking-Keel-Kontsevich's partial compactifications of cluster varieties as $F$-polynomials of projective representations of Jacobian algebras. Along the way, we show that both the finite-dimensional projective and the finite-dimensional injective representations of Jacobian algebras are well-behaved under Derksen-Weyman-Zelevinsky's mutations of representations.
		\end{abstract}
\maketitle

\tableofcontents

\section*{Introduction}

Cluster algebras, invented by Fomin-Zelevinsky in \cite{FZ}, are commutative algebras defined in terms of generators that are produced recursively through an algebraic-combinatorial operation called \emph{mutation}, which in turn also forces \emph{exchange relations} between such generators. Among the many examples of cluster algebras are homogeneous coordinate rings of Grassmannians and (partial) flag varieties. The main aim of Fomin-Zelevinsky's construction was to provide a combinatorial framework for the
study of the so-called canonical bases, which the mentioned examples admit. 

In \cite{FG} a geometric viewpoint on cluster algebras was taken where Fock-Goncharov attach two schemes to each cluster algebra, the cluster $\mathcal{A}$-variety which is closely related (and often equal) to the spectrum of the cluster algebra, and the cluster $\mathcal{X}$-variety which is in a sense dual to the cluster $\mathcal{A}$-variety. In \cite{GHKK}, under certain hypotheses, a canonical basis of the ring of regular functions on the $\mathcal{A}$-variety was constructed, parametrized by the tropical points of the $\mathcal{X}$-variety. When the cluster algebra has frozen vertices, Gross-Hacking-Keel-Kontsevich consider a partial compactification $\overline{\mathcal{A}}$ of the cluster $\mathcal{A}$-variety and a Landau-Ginzburg potential $W$ on the cluster $\mathcal{X}$-variety whose tropicalization determines the canonical basis of the ring of regular functions on $\overline{\mathcal{A}}$. 

In this way one obtains polyhedral cones defined by the tropicalization of $W$ expressed in certain toric charts and for each of these charts a toric degeneration of $\overline{\mathcal{A}}$. Among these cones appear many famous cones which are connected to the representation theory of algebraic groups, e.g. the Gelfand-Tsetlin cone (\cite{M2}), the Littlewood-Richardson cones (\cite{M}) and the string cones (\cite{BF, GKS20}).

The aim of this paper is give an interpretation of the Landau-Ginzburg potential $W$ in terms of representations of the Jacobian algebras defined on the underlying quivers. 

For each (decorated) representation $\mathcal{M}$ of a Jacobian algebra of $\Lambda$ we may define a polynomial, called the \emph{dual $F$-polynomial of $\mathcal{M}$}, as 
\begin{equation*}
F^{\vee}_{\mathcal{M}}(u_1,\ldots,u_m):=\sum_{\ee \in \mathbb{Z}_{\ge 0}^m} \chi(\Gr^{\ee}(\mathcal{M}))\displaystyle\prod_{i=1}^{m}u_i^{e_i},
\end{equation*}
where $\Gr^{\ee}(\mathcal{M})$ is the quiver Grassmannian of quotients of $\mathcal{M}$ and $\chi$ is the topological Euler-characteristic.

The following is our main result, which we prove under the hypothesis that every frozen variable has an optimized seed (see Definition \ref{potdefi}), and that the underlying quiver admits a non-degenerate potential.

\begin{Thm}[Theorem \ref{potasfpol}] The Landau-Ginzburg potential $W$, expressed in the toric chart to $Q$, is the sum of dual $F$-polynomials with constant terms removed, of the indecomposable projectives representations of a Jacobian algebra defined on $Q$.
\end{Thm}

We apply this theorem to the big reduced double Bruhat cells of simple simply-laced algebraic groups $G$ which have (up to codimension two) the structure of a cluster ${\mathcal{A}}$-variety whose partial compactification is the unipotent radical of a Borel of $G$. This leads to a quiver theoretic interpretation of the string cone inequalities. In type $A$ everything is understood explicitly in terms of combinatorics of certain paths in wiring diagrams.

In order to prove Theorem \ref{potasfpol} we study the mutation behavior of indecomposable projective finite-dimensional representations of Jacobian algebras (which are not assumed to be Jacobi-finite). They turn out to be well-behaved as the next theorem shows.

\begin{Thm}[Theorem \ref{thm:mut-of-proj-is-proj}] Suppose $(Q,S)$ is a non-degenerate quiver with potential. Let $k\ne \ell$ be vertices of $Q$ and set $(Q',S')$ to be the quiver with potential obtained by applying the $k^{th}$ mutation of quivers with potential to $(Q,S)$. If $\jacobalg{Q,S}e_{\ell}$ is finite dimensional then the $k^{th}$ mutation of representations, applied to the indecomposable projective $\jacobalg{Q,S}$-module $\jacobalg{Q,S}e_{\ell}$, is isomorphic to the indecomposable projective $\Lambda(Q',S')$-module $\Lambda(Q',S')e_{\ell}$.
\end{Thm}

In the Jacobi-finite case Theorem \ref{thm:mut-of-proj-is-proj} may also be deduced from \cite{P11} as explained in Remark \ref{Plamondon}. Note that our proof only uses methods from linear algebra.

Using the full power of Derksen-Weyman-Zelevinsky's machinery, the mutation behavior of projective representation allows us to show that the dual $F$-polynomial of an indecomposable projective representation at a frozen vertex is mutation invariant under $y$-seed mutation (see Theorem \ref{profFpol}).

The paper is organized as follows. In the first section we recall notions related to Jacobian algebras. In the second section we study the mutation behavior of projective representations of Jacobian algebras and prove Theorem \ref{thm:mut-of-proj-is-proj}. In the third section we study the mutation behavior of injective representations of Jacobian algebras. The fourth section introduces notions related to cluster algebras. In the fifth section we study the mutation behavior of $F$-polynomials of injective representations of Jacobian algebras and dual $F$-polynomials of projective representations of Jacobian algebras. In the sixth section we introduce cluster varieties, their partial compactification and Landau-Ginzburg potentials and prove Theorem \ref{potasfpol}. The last section deals with the example of big reduced double Bruhat cells and string cones.

\section*{Acknowledgements}
We are indebted to Markus Reineke for several insightful comments and discussions and to Pierre-Guy Plamondon for explaining us the results of \cite{P11}. We would furthermore like to thank Xin Fang for careful proofreading of some parts of this paper. The paper was written during a sabbatical visit of DLF to Sibylle Schroll at the Mathematisches Institut of the Universit\"at zu K\"oln. The great working conditions and the hospitality are gratefully acknowledged.

BL was supported by the SFB/TRR 191 'Symplectic Structures in Geometry, Algebra and Dynamics', funded by the DFG. DLF received support from UNAM's \emph{Dirección General de Asuntos del Personal Académico} through its \emph{Programa de Apoyos para la Superación del Personal Académico}.

\section{Background on quivers with potential and their representations}\label{backgroundQP}

\subsection{Quiver mutations}\label{subsec:quiver-mutations}

Recall that a \emph{quiver} is a finite directed graph, that is, a quadruple $Q=(Q_0,Q_1,h,t)$, where $Q_0$ is a set of
\emph{vertices}, $Q_1$ is the a set of \emph{arrows}, and $h:Q_1\rightarrow Q_0$ and $t:Q_1\rightarrow Q_0$ are the \emph{head}
and \emph{tail} functions. We write $a:i\rightarrow j$ to indicate that $a$ is an arrow of $Q$ with $t(a)=i$, $h(a)=j$. All of our quivers will be \emph{loop-free}, that is, there will never be any arrow $a$ with $t(a)=h(a)$.

A \emph{path of length} $d>0$ in $Q$ is a sequence $a_1a_2\ldots a_d$ of arrows with $t(a_j)=h(a_{j+1})$ for $j=1,\ldots,d-1$. A path
$a_1a_2\ldots a_d$ of length $d>0$ is a $d$\emph{-cycle} if $h(a_1)=t(a_d)$. A quiver is \emph{2-acyclic} if it has no 2-cycles.

Paths are composed as functions, that is, if $a=a_1\cdots a_d$ and $b=b_1\cdots b_{d'}$ are paths with $h(b)=t(a)$, then the concatenation $ab$ is defined as the path $a_1\cdots a_db_1\cdots b_{d'}$, starting at $t(b_{d'})$ and ending at $h(a_1)$. 

For $k\in Q_0$, a $k$\emph{-hook} in $Q$ is a path $ab$ of length 2 such that $a$ and $b$ are arrows such that $t(a)=k=h(b)$.

\begin{defi}\label{def:threesteps} Given a quiver $Q$ and a vertex $k\in Q_0$ such that $Q$ does not have $2$-cycles incident at $k$, we define the
\emph{mutation} of $Q$ with respect to $k$ as the quiver $\mu_k(Q)$ with vertex set $Q_0$ that results after applying the following three-step
procedure to $Q$:
\begin{itemize}
\item[(Step 1)] For each $k$-hook $ab$ introduce an arrow $[ab]:t(b)\rightarrow h(a)$;
\item[(Step 2)] replace each arrow $c$ incident to $k$ with an arrow $c^*$ in the opposite direction;
\item[(Step 3)] choose a maximal collection of disjoint 2-cycles and remove them.
\end{itemize}
The quiver obtained after applying only the first and second steps will be called the \emph{premutation} $\widetilde{\mu}_k(Q)$.
\end{defi}

\subsection{Quivers with potential and their mutations}

Our main reference for this subsection is \cite{DWZ1}.  

Given a quiver $Q$, we denote by $R$ the $\bbC$-vector space with basis $\{e_i\suchthat i\in Q_0\}$. If we define $e_ie_j=\delta_{ij}e_i$,
then $R$ becomes naturally a commutative semisimple $\bbC$-algebra, the \emph{vertex span} of $Q$; each idempotent $e_i$ is the
\emph{path of length zero} at $i$. We define the \emph{arrow span} of $Q$ as the $\bbC$-vector space $A$ with basis $Q_1$.
Notice that $A$ is an $R$-$R$-bimodule under $e_ia:=\delta_{i,h(a)}a$ and $ae_j:=a\delta_{t(a),j}$ for $i,j\in Q_0$ and $a\in Q_1$. For $d\geq 0$
let $A^d$ be the $\bbC$-vector space with basis all the paths of length $d$ in $Q$; it has a natural $R$-bimodule structure too and is isomorphic to the $d$-fold tensor product $A\otimes_R \cdots\otimes_RA$. Observe that $A^0=R$ and $A^1=A$.

\begin{defi} The \emph{complete path algebra} of $Q$ is the $\bbC$-vector space 
\begin{equation}\label{eq:comp-path-alg-def}
\RQ{Q}=\underset{d=0}{\overset{\infty}{\prod}}A^d;
\end{equation}
with multiplication 
\begin{equation}\label{eq:mult-in-comp-path-alg-in-terms-of-homog-components}
uv=\sum_{\ell\geq 0}\sum_{\ell_1+\ell_2=\ell}u^{(\ell_1)}v^{(\ell_2)}
\end{equation}
for any
$u=\sum_{\ell\geq 0}u^{(\ell)}$ and $v=\sum_{\ell\geq0}v^{(\ell)}$, with $u^{(\ell)},v^{(\ell)}\in A^\ell$ for every $\ell\geq 0$ (in \eqref{eq:mult-in-comp-path-alg-in-terms-of-homog-components}, the right-hand side is a well-defined element of $\RQ{Q}$: for $\ell\geq 0$ fixed, we have $\sum_{\ell_1+\ell_2=\ell}u^{(\ell_1)}v^{(\ell_2)}=\sum_{k=0}^\ell u^{(k)}v^{(k-\ell)}$).
\end{defi}

We see that the multiplication in $\RQ{Q}$ resembles the multiplication of formal power series and is ultimately induced by the concatenation of paths.

The complete path algebra $\RQ{Q}$ is a Hausdorff topological space under the $\idealM$-adic topology, whose fundamental system
of open neighborhoods around $0$ is given by the powers of $\idealM=\idealM(Q)=\prod_{d\geq 1}A^d$, which is the two-sided ideal of $\RQ{Q}$ generated by $Q_1$. A crucial property of this
topology is the following:
\begin{align*}
&\text{a sequence $(x_n)_{n\in\mathbb{Z}_{>0}}$ of elements of $\RQ{Q}$ converges if and only if for every $d\geq 0$},\\
\nonumber &\text{the sequence $(x_n^{(d)})_{n\in\mathbb{Z}_{>0}}$ stabilizes as $n\rightarrow\infty$, in which case
$\underset{n\rightarrow\infty}{\lim}x_n=\underset{d\geq 0}{\sum}\underset{n\rightarrow\infty}{\lim}x_n^{(d)}$,}
\end{align*}
where $x_n^{(d)}\in A^d$ denotes the degree-$d$ component of $x_n$ according to \eqref{eq:comp-path-alg-def}.

\begin{defi}\label{def:QP-stuff} Let $Q=(Q_0,Q_1,h,t)$ be a quiver.
\begin{itemize}\item A \emph{potential} on $Q$ is any element of $\RQ{Q}$ all of whose terms are cyclic paths of positive length;
\item two potentials $S,S'\in\RQ{Q}$ are \emph{cyclically equivalent} if their difference $S-S'$
lies in the $\idealM$-adic closure of the vector subspace of $\RQ{Q}$ spanned by all the elements of the form $a_1\ldots a_d-a_2\ldots a_da_1$ with $a_1\ldots
a_d$ a cyclic path of positive length in $Q$;
\item a \emph{quiver with potential} (or \emph{QP} for short) is a pair $(Q,S)$ consisting of a loop-free quiver $Q$ and a potential $S$ on $Q$;
\item The \emph{direct sum} of two QPs $(Q,S)$ and $(Q',S')$ that have the same vertex set $Q_0=Q_0'$ is the QP $(Q,S)\oplus(Q',S')=(Q\oplus Q',S+S')$, where $Q\oplus Q'$ is the quiver whose vertex set is $Q_0$ and whose arrow set is the disjoint union $Q_1\sqcup Q_1'$, with tail and head functions induced by those of $Q$ and $Q'$;
\item If $(Q,S)$ and $(Q',S')$ are QPs on the same set of vertices, we say that $(Q,S)$ is \emph{right-equivalent} to $(Q',S')$ if there exists a
\emph{right-equivalence} between them, that is, a $\bbC$-algebra isomorphism $\varphi:\RQ{Q}\rightarrow \RQ{Q'}$ that acts as the identity on the vertex span $R$ and such that $\varphi(S)$ is cyclically
equivalent to $S'$;
\item for each arrow $a\in Q_1$ and each cyclic path $a_1\ldots a_d$ in $Q$ we define the \emph{cyclic derivative}
\begin{equation}
\partial_a(a_1\ldots a_d)=\underset{i=1}{\overset{d}{\sum}}\delta_{a,a_i}a_{i+1}\ldots a_da_1\ldots a_{i-1},
\end{equation}
(where $\delta_{a,a_i}$ is the \emph{Kronecker delta}) and extend $\partial_a$ by linearity and continuity to have $\partial_a(S)\in\RQ{Q}$ defined for every potential $S$ on $Q$.
\item The \emph{Jacobian ideal} $J(S)$ is the $\idealM$-closure of the two-sided ideal of $\RQ{Q}$ generated by $\{\partial_a(S)\suchthat a\in Q_1\}$, and the
\emph{Jacobian algebra} $\jacobalg{Q,S}$ is the quotient algebra $\RQ{Q}/J(S)$.
\item A QP $(Q,S)$ is \emph{trivial} if $S\in A^2$ and $\{\partial_a(S)\suchthat a\in Q_1\}$ spans $A$ as $\bbC$-vector space.
\item A QP $(Q,S)$ is \emph{reduced} if the degree-$2$ component of $S$ is $0$, that is,
if the expression of $S$ involves no $2$-cycles.
\end{itemize}
\end{defi}

\begin{prop}\cite[Propositions 3.3 and 3.7]{DWZ1} \begin{enumerate}\item If two potentials $S,S'\in\RQ{Q}$ are cyclically equivalent, then
$\partial_a(S)=\partial_a(S')$ for all $a\in Q_1$.
\item If $\varphi:\RQ{Q}\rightarrow \RQ{Q'}$ is a right-equivalence between $(Q,S)$ and $(Q',S')$, then $\varphi$
sends $J(S)$ onto $J(S')$ and therefore induces an isomorphism $\overline{\varphi}:\jacobalg{Q,S}\rightarrow\jacobalg{Q',S'}$.
\end{enumerate}
\end{prop}

One of the main technical results of \cite{DWZ1} is the \emph{Splitting Theorem}, which we now state.

\begin{thm}\cite[Theorem 4.6]{DWZ1}\label{thm:splittingthm} For every QP $(Q,S)$ there exist a reduced QP $(Q_{\operatorname{red}},S_{\operatorname{red}})$ and a trivial QP
$(Q_{\operatorname{triv}},S_{\operatorname{triv}})$ such that $(Q,S)$ is right-equivalent to the direct sum $(Q_{\operatorname{red}},S_{\operatorname{red}})\oplus(Q_{\operatorname{triv}},S_{\operatorname{triv}})$. Furthermore,
the right-equivalence class of each of the QPs $(Q_{\operatorname{red}},S_{\operatorname{red}})$ and  $(Q_{\operatorname{triv}},S_{\operatorname{triv}})$ is determined by the right-equivalence class of $(Q,S)$.
\end{thm}

The QPs $(Q_{\operatorname{red}},S_{\operatorname{red}})$ and  $(Q_{\operatorname{triv}},S_{\operatorname{triv}})$ are called the \emph{reduced part} and the \emph{trivial
part} of $(Q,S)$, respectively; this terminology is well defined up to right-equivalence.

We now turn to the definition of mutation of a QP. Let $(Q,S)$ be a QP on the vertex set $Q_0$ and take $k\in Q_0$. Assume that $Q$ does not have
2-cycles incident to $k$. If necessary, replace $S$ with a cyclically equivalent potential so that every cyclic path appearing in the expression of $S$ begins and ends at a vertex different from $k$. This allows us to define $[S]\in\RQ{\widetilde{\mu}_k(Q)}$ as the potential on obtained from
$S$ by replacing each $k$-hook $ab$ with the arrow $[ab]$ (see the line preceding Definition \ref{def:threesteps}). Furthermore, set $\Delta_k(Q):=\sum a^*[ab]b^*$,
where the sum runs over all $k$-hooks $ab$ of $Q$.

\begin{defi}\label{def:QP-mutation}\cite[(5.8),(5.9) and Definition 5.5]{DWZ1} Under the assumptions and notation just stated, we define the \emph{premutation} of $(Q,S)$ with respect to $k$ as the QP
$\widetilde{\mu}_k(Q,S)=(\widetilde{\mu}_k(Q),\widetilde{\mu}_k(S))$ (see Definition \ref{def:threesteps}), where
$\widetilde{\mu}_k(S):=[S]+\Delta_k(Q)$. The \emph{mutation} $\mu_k(Q,S)$ of $(Q,S)$ is then defined to be reduced part of
$\widetilde{\mu}_k(Q,S)$.
\end{defi}

\begin{defi}\cite[Definition 7.2]{DWZ1}\label{Def:nondegpot} 
A QP $(Q,S)$ is \emph{non-degenerate} if it is 2-acyclic and the QP obtained after any possible sequence of QP-mutations is 2-acyclic.
\end{defi}

\begin{thm}\label{thm:propertiesofQP-mutations}\cite[Theorem 5.2, Corollary 5.4, Theorem 5.7, Corollary 7.4]{DWZ1}
\begin{enumerate}\item Premutations and mutations of QPs are well defined up to right-equivalence.
\item Mutations of QPs are involutive up to right-equivalence.
\item Every 2-acyclic quiver admits a non-degenerate potential over $\bbC$.
\end{enumerate}
\end{thm}

\subsection{Mutations of representations}

As in the previous subsection, our main reference is \cite{DWZ1}.

\begin{defi}\cite[Section 2]{MRZ}\cite[Definition 10.1]{DWZ1} Let $(Q,S)$ be a QP. A \emph{decorated representation} of $(Q,S)$, or of its Jacobian algebra $\jacobalg{Q,S}$, is a pair $\mathcal{M}=(M,\mathbf{v})$, where $M$ is a finite-dimensional left $\jacobalg{Q,S}$-module and $\mathbf{v}=(v_j)_{j\in Q_0}$ is a $Q_0$-tuple of non-negative integers.
\end{defi}

By setting $M_j=e_jM$ for each $j\in Q_0$, and $a_M:M_{t(a)}\rightarrow M_{h(a)}$ as the multiplication by $a\in Q_1$ given by the $\jacobalg{Q,S}$-module structure of $M$, we easily see that each $\jacobalg{Q,S}$-module induces a representation of the quiver $Q$. Furthermore, since every finite-dimensional $\RQ{Q}$-module is nilpotent (that is, annihilated by some power of $\idealM$, see \cite[\S10]{DWZ1}) any finite-dimensional left $\jacobalg{Q,S}$-module is prescribed by the following data:
\begin{enumerate}\item A tuple $(M_i)_{i\in Q_0}$ of finite-dimensional $\bbC$-vector spaces;
\item a family $(a_M:M_{t(a)}\rightarrow M_{h(a)})_{a\in Q_0}$ of $\bbC$-linear transformations annihilated by $\{\partial_a(S)\suchthat a\in Q_1\}$, for which there exists an integer $r\geq 1$ with the property that the composition ${a_1}_{M}\ldots {a_r}_{M}$ is identically zero for every $r$-path $a_1\ldots a_r$ in $Q$.
\end{enumerate}

Let $(Q,S)$ be
any QP, and let $\varphi:\RQ{ Q_{\operatorname{red}}\oplus C}\rightarrow\RQ{Q}$ be a right equivalence between a direct sum 
$(Q_{\operatorname{red}},S_{\operatorname{red}})\oplus(C,T)$ and $(Q,S)$, where $(Q_{\operatorname{red}},S_{\operatorname{red}})$ is a reduced QP and $(C,T)$ is a trivial QP. Denoting by $\iota:\RQ{Q_{\operatorname{red}}}\rightarrow \RQ{Q_{\operatorname{red}}\oplus C}$ the canonical inclusion, by \cite[Propositions 3.7 and 4.5]{DWZ1} we have induced $\bbC$-algebra isomorphisms
$$
\xymatrix{
\jacobalg{Q_{\operatorname{red}},S_{\operatorname{red}}}\ar[r]^{\overline{\iota}\qquad\quad} & \jacobalg{(Q_{\operatorname{red}}\oplus C,S_{\operatorname{red}}+T} \ar[r]^{\qquad\quad \overline{\varphi}} &\jacobalg{Q,S}.
}
$$
Thus, if
$\mathcal{M}=(M,\mathbf{v})$ is a decorated representation of $(Q,S)$, and we set $M^\varphi=M$ as $\bbC$-vector space, and define an action of $R\langle\langle
Q_{\operatorname{red}}\rangle\rangle$ on $M^\varphi$ by setting $u_{M^\varphi}:=\varphi(u)_M$ for $u\in\RQ{
Q_{\operatorname{red}}}\rangle\rangle$, then $\mathcal{M}_{\operatorname{red}}:=(M^\varphi,\mathbf{v})$ is a decorated representation of $(Q_{\operatorname{red}},S_{\operatorname{red}})$.

\begin{defi}\label{def:reducedpartofrep}\cite[Definition 10.4]{DWZ1}
The decorated representation $\mathcal{M}_{\operatorname{red}}=(M^\varphi,\mathbf{v})$ of $(Q_{\operatorname{red}},S_{\operatorname{red}})$ is the \emph{reduced part} of $\mathcal{M}$.
\end{defi}

We now turn to the representation-theoretic analogue of the notion of QP-mutation. Let $(Q,S)$ be a QP. Fix a vertex $k\in Q_0$, and suppose that $Q$ does not have 2-cycles incident to $k$. Suppose that $a_1,\ldots,a_s$ (resp. $b_1,\ldots,b_r$) are the arrows of $Q$ ending at $k$ (resp. starting at $k$). Take a decorated representation $\mathcal{M}=(M,\mathbf{v})$ of $(Q,S)$ and set
$$
M_{\operatorname{in}}=M_{\operatorname{in}}(k)=\bigoplus_{i=1}^sM_{t(a_i)}, \ \ M_{\operatorname{out}}=M_{\operatorname{out}}(k)=\bigoplus_{j=1}^rM_{h(b_j)}.
$$
The action of the arrows $a_1,\ldots,a_s$ and $b_1,\ldots,b_t$ on $M$ induces $K$-linear maps
$$
\alpha=\alpha_k=[(a_1)_M \ \ldots \ (a_s)_M]:M_{\operatorname{in}}\rightarrow M_k, \ \ \beta=\beta_k=\left[\begin{array}{c}(b_1)_M \\ \vdots \\ (b_r)_M\end{array} \right]:M_k\rightarrow M_{\operatorname{out}}.
$$

For each pair $(i,j)\in [1,s]\times [1,r]$, set $\gamma_{i,j}:=\partial_{b_ja_i}(S)_M:M_{h(b_j)}\rightarrow M_{t(a_i)}$, and arrange these maps into a matrix to obtain a linear map $\gamma=\gamma_k:M_{\operatorname{out}}\rightarrow M_{\operatorname{in}}$. Since $M$ is a $\jacobalg{Q,S}$-module, by \cite[Lemma 10.6]{DWZ1} we have $\alpha\gamma=0$ and $\gamma\beta=0$.

For $j\in Q_0$ define
$$
\overline{M}_j:=\begin{cases}
    M_j & \text{if} \ j\neq k;\\
    \coker\beta\oplus\frac{\ker\alpha}{\image\gamma}\oplus \bbC^{v_k} & \text{if} \ j=k;
\end{cases}
\qquad 
\overline{v}_j:=\begin{cases}
v_j & \text{if} \ j\neq k;\\
\dim_{\bbC}\left(\frac{\ker\beta}{\ker\beta\cap\image\alpha}\right) & \text{if} \ j=k.
\end{cases}
$$

We define the action of the arrows of $\widetilde{\mu}_k(Q)$ on $\overline{M}:=\bigoplus_{j\in Q_0}\overline{M}_j$ as follows. If $c$ is an arrow of $Q$ not incident to $k$, we define $c_{\overline{M}}:=c_M$, and for each pair $(i,j)\in[1,s]\times [1,r]$ we set $[b_ja_i]_{\overline{M}}:=(b_ja_i)_M={b_j}_M{a_i}_M$. To define the action of the remaining arrows, choose a linear map $\sigma:\frac{\ker\alpha}{\image\gamma}\rightarrow\ker\alpha$ such that the composition $\frac{\ker\alpha}{\image\gamma}\overset{\sigma}{\rightarrow}\ker\alpha\twoheadrightarrow\frac{\ker\alpha}{\image\gamma}$ is the identity (where $\ker\alpha\twoheadrightarrow\frac{\ker\alpha}{\image\gamma}$ is the natural projection). Then set
$$
 \left[\begin{array}{c}(a_1^*)_{\overline{M}} \\ \vdots \\ (a_s^*)_{\overline{M}}\end{array} \right]:=\overline{\beta}:=\left[\begin{array}{ccc}\overline{\gamma} & \iota\sigma & 0\end{array}\right]:\overline{M}_k\rightarrow M_{\operatorname{in}}, \qquad
 \left[\begin{array}{ccc}(b_1^*)_{\overline{M}}  \ldots & (b_r^*)_{\overline{M}}\end{array}\right]:=\overline{\alpha}:=\left[\begin{array}{c}-p \\ 0 \\ 0\end{array}\right]:M_{\operatorname{out}}\rightarrow\overline{M}_k,
$$
where $\iota:\ker\alpha\rightarrow M_{\operatorname{in}}$ is the inclusion, $p:M_{\operatorname{out}}\rightarrow\coker\beta$ is the natural projection, and $\overline{\gamma}:\coker\beta\rightarrow M_{\operatorname{in}}$ is the map induced by $\gamma$ (recall that $\image\beta\subseteq\ker\gamma$).

Since $\idealM^{\ell}M=0$ for some sufficiently large $\ell$, this action of the arrows of $\widetilde{\mu}_k(Q)$ on $\overline{M}$ extends uniquely to an action of $\RQ{\widetilde{\mu}_k(Q)}$ under which $\overline{M}$ is an $\RQ{\widetilde{\mu}_k(Q)}$-module. By \cite[Proposition 10.7]{DWZ1}, $\overline{M}$ is a left $\jacobalg{\widetilde{\mu}_k(Q),\widetilde{\mu}_k(S)}$-module.

\begin{remark} The choice of the linear map $\sigma$ is not canonical. However, by \cite[Proposition 10.9]{DWZ1}, different choices lead to isomorphic $\RQ{\widetilde{\mu}_k(Q)}$-module structures on $\overline{M}$.
\end{remark}

\begin{defi}\cite[Definition 10.22]{DWZ1} The pair $\widetilde{\mu}_k(\mathcal{M}):=(\overline{M},\overline{\mathbf{v}})$ is called the \emph{premutation}  of $\mathcal{M}=(M,\mathbf{v})$ at $k$. The \emph{mutation} of $\mathcal{M}$ at $k$, denoted by $\mu_k(\mathcal{M})$, is the reduced part of $\widetilde{\mu}_k(\mathcal{M})$.
\end{defi}

\begin{thm}\cite[Proposition 10.10, Corollary 10.12, and Theorem 10.13]{DWZ1}
\begin{enumerate}\item Premutations and mutations of decorated representations are well defined up to right-equivalence.
\item Mutations of decorated representations are involutive up to right-equivalence.
\end{enumerate}
\end{thm}

\section{Behavior of projectives under mutations of representations}

The aim of this section is to prove the following. Notice that $(Q,S)$ is not assumed to be Jacobi-finite.

\begin{thm}\label{thm:mut-of-proj-is-proj} For any quiver with potential $(Q,S)$ and any pair of distinct vertices $k,\ell\in Q_0$, with $k$ not  incident to $2$-cycles of $Q$, if the indecomposable projective $P_{(Q,S)}(\ell)$ is finite-dimensional, then so is the indecomposable projective $P_{\mu_k(Q,S)}(\ell)$, and the decorated representations $\mu_k(P_{(Q,S)}(\ell),0)$ and $(P_{\mu_k(Q,S)}(\ell),0)$ of the Jacobian algebra $\jacobalg{\mu_k(Q,S)}$ are isomorphic.
\end{thm}

\begin{remark}\label{Plamondon}
Theorem \ref{thm:mut-of-proj-is-proj} can be deduced from \cite[Proposition 4.1]{P11}. Indeed, with the notation from \cite{P11}, the functor $F$ therein sends the object $\Sigma^{-1}\Gamma_{\ell}$ to the projective $P_{(Q,S)}(\ell)$ (which is assumed to be finite-dimensional). Since mutation at $k$ sends $\Gamma_\ell$ to $\Gamma'_\ell$ if $\ell\neq k$, one deduces Theorem \ref{thm:mut-of-proj-is-proj}. Similarly, Corollary \ref{coro:mut-of-inj-is-inj} below can be deduced from \cite[Proposition 4.1]{P11} too. We thank Pierre-Guy Plamondon for pointing this out to us. Our proofs of Theorem \ref{thm:mut-of-proj-is-proj} and Corollary \ref{coro:mut-of-inj-is-inj} are, however, fully independent from those in \cite{P11} and rely only on the formal path combinatorics and linear algebra within complete path algebras.
\end{remark}

We will mantain the assumptions from Theorem \ref{thm:mut-of-proj-is-proj} throughout the whole section. We will suppose that the arrows ending (resp. starting) at $k$ are $a_1,\ldots,a_s$ (resp. $b_1,\ldots,b_r$), and that the arrows from $\ell$ to $k$ (resp. from $k$ to $\ell$) are precisely $a_1,\ldots,a_{\overline{s}}$ (resp. $b_1,\ldots,b_{\overline{r}}$). Note that $\overline{s}$ and $\overline{r}$ cannot simultaneously be non-zero, since $Q$ does not have $2$-cycles incident to~$k$.

As a preparation for the proof of Theorem \ref{thm:mut-of-proj-is-proj}, we need to study some of the vector spaces that arise when one applies the mutation of representations to the radical $\rad_{(Q,S)}P_{(Q,S)}(\ell)$.

Let us first describe the representation $\rad_{(Q,S)}P_{(Q,S)}(\ell)$. Let $\idealM$ be the two-sided ideal of $\RQ{Q}$ generated by the arrows of $Q$, and let $\pi:\RQ{Q}\rightarrow\jacobalg{Q,S}$ denote the projection from the complete path algebra to the Jacobian algebra. Then $\rad_{(Q,S)}P_{(Q,S)}(\ell)=\pi(\idealM)e_{\ell}$ (even if $P_{(Q,S)}(\ell)$ is infinite-dimensional over $\mathbb{C}$). Furthermore, as a representation of $Q$, to a vertex $j$ it attaches the space $e_j\pi(\idealM)e_{\ell}$, and for an arrow $a:j\rightarrow i$, the corresponding linear map $e_j\pi(\idealM)e_{\ell}\rightarrow e_i\pi(\idealM)e_{\ell}$ is induced by merely composing with $a$ each path from $\ell$ to $j$, i.e., $p\mapsto a\cdot p$.

If one is interested in applying the $k^{\operatorname{th}}$ mutation of representations to $\rad_{(Q,S)}P_{(Q,S)}(\ell)=\pi(\idealM)e_{\ell}$, one has to form Derksen-Weyman-Zelevinsky's $\alpha$-$\beta$-$\gamma$ triangle of linear maps
$$
\xymatrix{
 & e_k\pi(\idealM)e_\ell  \ar[dr]^{\beta:=\left[\begin{array}{c}b_1\cdot \\ \vdots \\ b_r\cdot \end{array}\right]} & \\
 \bigoplus_{a\in Q_1:h(a)=k} e_{t(a)} \pi(\idealM) e_\ell \ar[ur]^{\alpha:=\left[\begin{array}{ccc}a_1\cdot  & \cdots & a_s\cdot \end{array}\right]\qquad} & & \bigoplus_{b\in Q_1:t(b)=k} e_{h(b)} \pi(\idealM)e_\ell \ar[ll]^{\gamma:=\text{ {\tiny $\left[\begin{array}{ccc}
 \partial_{b_1a_1}(S) \cdot & \cdots & \partial_{b_r a_1}(S)\cdot \\
 \vdots & \ddots & \vdots\\
 \partial_{b_1a_s}(S) \cdot & \cdots & \partial_{b_ra_s}(S)\cdot
 \end{array}\right]$}}}
 }
$$
By \cite[Lemma 10.6]{DWZ1}, $\image\gamma_k\subseteq\ker\alpha_k$.
Our first result specifies an explicit spanning set for the vector space $\ker\alpha_k/\image\gamma_k$.

\begin{prop}\label{prop:spanning ker a/ker gamma for radical} The elements $(\partial_{b_1a_i}(S))_{i=1}^s,\ldots,(\partial_{b_{\bar{r}}a_i}(S))_{i=1}^s\in \bigoplus_{i=1}^s e_{t(a_i)}\pi(\idealM) e_\ell$ belong to $\ker\alpha_k$, and their projections $(\partial_{b_1a_i}(S))_{i=1}^s+\image\gamma_k,\ldots,(\partial_{b_{\bar{r}}a_i}(S))_{i=1}^s+\image\gamma_k$ span
$\ker\alpha_k/\image\gamma_k$ as a $\bbC$-vector space.
\end{prop}

\begin{proof} It is obvious that the elements $(\partial_{b_1a_i}(S))_{i=1}^s,\ldots,(\partial_{b_{\bar{r}}a_i}(S))_{i=1}^s\in \bigoplus_{i=1}^s e_{t(a_i)}\pi(\idealM) e_\ell$ belong to $\ker\alpha_k$.

Take any element $(p_1,\ldots,p_s)\in\bigoplus_{a\in Q_1:h(a)=k} e_{t(a)} \idealM e_\ell$ such that $a_1p_1+\cdots+a_sp_s\in J(S)$, so it is the $\idealM$-adic limit of some sequence $\left(\xi_n\right)_{n\in\bbZ_{>0}}$ of elements
$$
\xi_n
=
\sum_{\omega\in Q_1}\sum_{q}\lambda_{n,\omega,q}\partial_{\omega}(S)\rho_{n,\omega,q}
$$
of $\RQ{Q}$,
with $\lambda_{n,\omega,q}\in e_k\RQ{Q}e_{t(\omega)}$ and $\rho_{n,\omega,q}\in e_{h(\omega)}\RQ{Q}e_{\ell}$, where each sum $\sum_{q}\lambda_{n,\omega,q}\partial_{\omega}(S)\rho_{n,\omega,q}$ has only finitely many summands.

For each $n$, each $\omega\in Q_1$ and each $q$, write
$$
\lambda_{n,\omega,q}=\lambda_{n,\omega,q}^{(0)}+ \lambda_{n,\omega,q}^{(>0)}
$$
with $\lambda_{n,\omega,q}^{(0)}=x_{n,\omega,q}e_k\in\mathbb{C}e_k$ the degree-$0$ component of $\lambda_{n,\omega,q}$ in the complete path algebra $\RQ{Q}$, and $\lambda_{n,\omega,q}^{(>0)}:=\lambda_{n,\omega,q}-\lambda_{n,\omega,q}^{(0)}\in e_k\idealM e_{t(\omega)}$. This way,
$$
\xi_n=
\sum_{\omega\in Q_1}\sum_q\lambda_{n,\omega,q}\partial_{\omega}(S)\rho_{n,\omega,q}=
\sum_{\omega\in Q_1}\sum_q\left(\lambda_{n,\omega,q}^{(0)}\partial_{\omega}(S)\rho_{n,\omega,q}+
\lambda_{n,\omega,q}^{(>0)}\partial_{\omega}(S)\rho_{n,\omega,q}\right).
$$
For each term of the form $\lambda_{n,\omega,q}^{(0)}\partial_{\omega}(S)\rho_{n,\omega,q}$ with $\lambda_{n,\omega,q}^{(0)}\neq 0$
we have $k=t(\omega)$, hence $\omega\in\{b_1,\ldots,b_r\}$, say $\omega=b_j$, and 
$$
\lambda_{n,\omega,q}^{(0)}\partial_{\omega}(S)\rho_{n,\omega,q}=
\lambda_{n,b_j,q}^{(0)}\partial_{b_j}(S)\rho_{n,b_j,q}=\lambda_{n,b_j,q}^{(0)}\sum_{i=1}^sa_i\partial_{b_ja_i}(S)\rho_{n,b_j,q}.
$$
Furthermore, since $\rho_{n,\omega,q}\in e_{h(\omega)}\RQ{Q}e_{\ell}$, we see that the degree-$0$ component of $\rho_{n,\omega,q}$ in the complete path algebra $\RQ{Q}$ is $0$ for $j=\overline{r}+1,\ldots,r$. 
Thus, denoting by $\rho_{n,b_j,q}^{(0)}:=y_{n,b_j,q}e_{h(b_j)}\in\bbC e_{h(b_j)}$ the degree-$0$ component of $\rho_{n,b_j,q}$ in the complete path algebra $\RQ{Q}$ for every $j=1,\ldots,r$, and writing $\rho_{n,b_j,q}^{(>0)}:=\rho_{n,b_j,q}-\rho_{n,b_j,q}^{(0)}$, for $j=\overline{r}+1,\ldots,r$ we have $\rho_{n,b_j,q}^{(0)}=0$.

On the other hand, for every $\omega\in Q_1$, the element $\lambda_{n,\omega,q}^{(>0)}$ can be written as $\lambda_{n,\omega,q}^{(>0)}=\sum_{i=1}^s a_i\nu_{n,\omega,q,i}$ with $\nu_{n,\omega,q,i}\in e_{t(a_i)}\RQ{Q}e_{t(\omega)}$. And for $\omega\in Q_1\setminus\{b_1,\ldots,b_r\}$, we have $t(\omega)\neq k$, hence $\lambda_{n,\omega,q}=\lambda_{n,\omega,q}^{(>0)}$.
Therefore, 
\begin{align*}
\xi_n
&=
\sum_{j=1}^{r}\sum_q\left(\lambda^{(0)}_{n,b_j,q}\sum_{i=1}^sa_i\partial_{b_ja_i}(S)\rho_{n,b_j,q}+\sum_{i=1}^s a_i\nu_{n,b_j,q,i}\partial_{b_j}(S)\rho_{n,b_j,q}\right)
+
\sum_{\omega\in Q_1\setminus\{b_1,\ldots,b_r\}}\sum_q\left(\sum_{i=1}^s a_i\nu_{n,\omega,q,i}\partial_{\omega}(S)\rho_{n,\omega,q}\right)\\
&=
\sum_{i=1}^sa_i\left(\sum_{j=1}^{r}\sum_q\left(x_{n,b_j,q}\partial_{b_ja_i}(S)\rho_{n,b_j,q}+\nu_{n,b_j,q,i}\partial_{b_j}(S)\rho_{n,b_j,q}\right)
+
\sum_{\omega\in Q_1\setminus\{b_1,\ldots,b_r\}}\sum_q\left(\nu_{n,\omega,q,i}\partial_{\omega}(S)\rho_{n,\omega,q}\right)\right)\\
&=
\sum_{i=1}^sa_i\left(\sum_{j=1}^{r}\sum_q\left(x_{n,b_j,q}\partial_{b_ja_i}(S)\rho_{n,b_j,q}\right)
+
\sum_{\omega\in Q_1}\sum_q\left(\nu_{n,\omega,q,i}\partial_{\omega}(S)\rho_{n,\omega,q}\right)\right)\\
&=
\sum_{i=1}^sa_i\left(\sum_{j=1}^{r}\sum_q\left(x_{n,b_j,q}\partial_{b_ja_i}(S)(\rho_{n,b_j,q}^{(0)}+\rho_{n,b_j,q}^{(>0)})\right)
+
\sum_{\omega\in Q_1}\sum_q\left(\nu_{n,\omega,q,i}\partial_{\omega}(S)\rho_{n,\omega,q}\right)\right)\\
&=
\sum_{i=1}^sa_i\left(\sum_{j=1}^{r}\sum_q\left(x_{n,b_j,q}\partial_{b_ja_i}(S)\rho_{n,b_j,q}^{(0)}+x_{n,b_j,q}\partial_{b_ja_i}(S)\rho_{n,b_j,q}^{(>0)}\right)
+
\sum_{\omega\in Q_1}\sum_q\left(\nu_{n,\omega,q,i}\partial_{\omega}(S)\rho_{n,\omega,q}\right)\right)\\
&=
\sum_{i=1}^sa_i\left(\sum_{j=1}^{r}\sum_q\left(x_{n,b_j,q}\partial_{b_ja_i}(S)\rho_{n,b_j,q}^{(0)}\right)+\sum_{j=1}^{r}\sum_q\left(x_{n,b_j,q}\partial_{b_ja_i}(S)\rho_{n,b_j,q}^{(>0)}\right)
+
\sum_{\omega\in Q_1}\sum_q\left(\nu_{n,\omega,q,i}\partial_{\omega}(S)\rho_{n,\omega,q}\right)\right)\\
&=
\sum_{i=1}^sa_i\left(\sum_{j=1}^{\overline{r}}\sum_q\left(x_{n,b_j,q}y_{n,b_j,q}\partial_{b_ja_i}(S)\right)+\sum_{j=1}^{r}\sum_q\left(x_{n,b_j,q}\partial_{b_ja_i}(S)\rho_{n,b_j,q}^{(>0)}\right)
+
\sum_{\omega\in Q_1}\sum_q\left(\nu_{n,\omega,q,i}\partial_{\omega}(S)\rho_{n,\omega,q}\right)\right).
\end{align*}

Take any $\varepsilon\in\bbZ_{>0}$. Then there exists $N\in\bbZ_{>0}$ such that for all $n>N$ we have
$$
a_1p_1+\cdots+a_sp_s-\xi_n\in\mathfrak{m}^\varepsilon,
$$
and hence, for every $i=1,\ldots,s$, we have
$$
p_i-\left(\sum_{j=1}^{\overline{r}}\sum_q\left(x_{n,b_j,q}y_{n,b_j,q}\partial_{b_ja_i}(S)\right)+\sum_{j=1}^{r}\sum_q\left(x_{n,b_j,q}\partial_{b_ja_i}(S)\rho_{n,b_j,q}^{(>0)}\right)
+
\sum_{\omega\in Q_1}\sum_q\left(\nu_{n,\omega,q,i}\partial_{\omega}(S)\rho_{n,\omega,q}\right)\right)\in\idealM^{\varepsilon-1}.
$$

Thus, the vector 
$$
\left(\begin{array}{c}p_1\\ \vdots \\
p_i\\
\vdots \\ p_s\end{array}\right)
- 
\sum_{j=1}^{\overline{r}}\sum_qx_{n,b_j,q}y_{n,b_j,q}\left(\begin{array}{c}\partial_{b_ja_1}(S)\\
\vdots \\
\partial_{b_ja_i}(S)\\
\vdots \\
\partial_{b_ja_s}(S)
\end{array}\right)
-
\sum_{j=1}^{r}\sum_qx_{n,b_j,q}\left(\begin{array}{c}\partial_{b_ja_1}(S)\rho_{n,b_j,q}^{(>0)}\\
\vdots \\
\partial_{b_ja_i}(S)\rho_{n,b_j,q}^{(>0)}\\
\vdots\\
\partial_{b_ja_s}(S)\rho_{n,b_j,q}^{(>0)}\end{array}\right)
-
\sum_{\omega\in Q_1}\sum_q\left(\begin{array}{c}\nu_{n,\omega,q,1}\partial_{\omega}(S)\rho_{n,\omega,q}\\
\vdots\\
\nu_{n,\omega,q,i}\partial_{\omega}(S)\rho_{n,\omega,q}\\
\vdots\\
\nu_{n,\omega,q,s}\partial_{\omega}(S)\rho_{n,\omega,q}\end{array}\right),
$$
belongs to $\underset{s}{\underbrace{\idealM^{\varepsilon-1}\times \cdots\times \idealM^{\varepsilon-1}}}$, 
which means that in $\bigoplus_{a\in Q_1:h(a)=k} e_{t(a)} \idealM e_\ell$,
the sequence
$$
\left(\sum_{j=1}^{\overline{r}}\sum_qx_{n,b_j,q}y_{n,b_j,q}\left(\begin{array}{c}\partial_{b_ja_1}(S)\\
\vdots \\
\partial_{b_ja_i}(S)\\
\vdots \\
\partial_{b_ja_s}(S)
\end{array}\right)
+
\sum_{j=1}^{r}\sum_qx_{n,b_j,q}\left(\begin{array}{c}\partial_{b_ja_1}(S)\rho_{n,b_j,q}^{(>0)}\\
\vdots \\
\partial_{b_ja_i}(S)\rho_{n,b_j,q}^{(>0)}\\
\vdots\\
\partial_{b_ja_s}(S)\rho_{n,b_j,q}^{(>0)}\end{array}\right)
+
\sum_{\omega\in Q_1}\sum_q\left(\begin{array}{c}\nu_{n,\omega,q,1}\partial_{\omega}(S)\rho_{n,\omega,q}\\
\vdots\\
\nu_{n,\omega,q,i}\partial_{\omega}(S)\rho_{n,\omega,q}\\
\vdots\\
\nu_{n,\omega,q,s}\partial_{\omega}(S)\rho_{n,\omega,q}\end{array}\right)\right)_{n\in\mathbb{Z}_{>0}}
$$
converges to 
\begin{equation}\label{eq:vector-of-ps}
\left(\begin{array}{c}p_1\\ \vdots \\
p_i\\
\vdots \\ p_s\end{array}\right).
\end{equation}

Now, the $\bbC$-linear map
\begin{equation}\label{eq:gamma-for-Jac-alg}
\left[\begin{array}{ccc}
 \partial_{b_1a_1}(S) \cdot & \cdots & \partial_{b_r a_1}(S)\cdot \\
 \vdots & \ddots & \vdots\\
 \partial_{b_1a_s}(S) \cdot & \cdots & \partial_{b_ra_s}(S)\cdot
 \end{array}\right]:\bigoplus_{b\in Q_1:t(b)=k} e_{h(b)} \idealM e_\ell\rightarrow \bigoplus_{a\in Q_1:h(a)=k} e_{t(a)} \idealM e_\ell
\end{equation}
is continuous, so by \cite[Lemma 13.6]{DWZ1}, its image is a closed $\bbC$-vector subspace of $\bigoplus_{a\in Q_1:h(a)=k} e_{t(a)} \idealM e_\ell$. On the other hand, $\bigoplus_{a\in Q_1:h(a)=k} e_{t(a)} J(S) e_\ell$ is a closed $\bbC$-vector subspace of $\bigoplus_{a\in Q_1:h(a)=k} e_{t(a)} \idealM e_\ell$, whereas the $\bbC$-vector subspace of $\bigoplus_{a\in Q_1:h(a)=k} e_{t(a)} \idealM e_\ell$ spanned by the set
\begin{equation}\label{eq:set-of-2nd-order-cyclic-derivatives}
\left\{\left(\begin{array}{c}\partial_{b_1a_1}(S)\\
\vdots \\
\partial_{b_1a_i}(S)\\
\vdots \\
\partial_{b_1a_s}(S)
\end{array}\right),\ldots,\left(\begin{array}{c}\partial_{b_{\overline{r}}a_1}(S)\\
\vdots \\
\partial_{b_{\overline{r}}a_i}(S)\\
\vdots \\
\partial_{b_{\overline{r}}a_s}(S)
\end{array}\right)\right\}
\end{equation}
is closed by \cite[Lemma 13.2]{DWZ1}.

Therefore, by \cite[Lemma 13.4]{DWZ1}, the $\bbC$-vector subspace of $\bigoplus_{a\in Q_1:h(a)=k} e_{t(a)} \idealM e_\ell$ spanned by the union of $\bigoplus_{a\in Q_1:h(a)=k} e_{t(a)} J(S) e_\ell$, the image of the map \eqref{eq:gamma-for-Jac-alg}, and the set \eqref{eq:set-of-2nd-order-cyclic-derivatives}, is closed. It hence contains the vector \eqref{eq:vector-of-ps}. This finishes the proof of Proposition \ref{prop:spanning ker a/ker gamma for radical}.
\end{proof}

Set $\widetilde{\idealM}$ to be the $2$-sided ideal of $\RQ{\widetilde{\mu}_k(Q)}$ generated by the arrows of $\widetilde{\mu}_k(Q)$, and denote by $\widetilde{\pi}:\RQ{\widetilde{\mu}_k(Q)}\rightarrow\jacobalg{\widetilde{\mu}_k(Q),\widetilde{\mu}_k(S)}$ the projection to the Jacobian algebra.

Consider the $\bbC$-linear maps
\begin{align}\label{eq:bracket-isomorphisms}
& [-]:\bigoplus_{a\in Q_1:h(a)=k} e_{t(a)}\pi(\idealM) e_\ell\rightarrow \bigoplus_{a\in Q_1:h(a)=k} e_{t(a)}\cdot \widetilde{\pi}(\widetilde{\idealM})\cdot e_\ell\\
\nonumber
\text{and} \ \ & [-]:\bigoplus_{b\in Q_1:t(b)=k} e_{h(b)}\pi(\idealM) e_\ell\rightarrow \bigoplus_{b\in Q_1:t(b)=k} e_{h(b)}\cdot \widetilde{\pi}(\widetilde{\idealM})\cdot e_\ell
\end{align}
given by replacing each 2-path $b_{j}a_i$ of $Q$ with the arrow $[b_ja_i]$ of $\widetilde{\mu}_k(Q)$ in the expression of each path of $Q$ not starting nor ending at $k$. By \cite[Proposition~6.1]{DWZ1}, these two maps are bijective. In the proof of the following proposition we will think of them as identifications.

\begin{prop}\label{prop:alpha-beta-gamma-maps-for-radicals}
The $\bbC$-linear maps
$$
\xymatrix{
 & e_k\pi(\idealM)e_\ell  \ar[dr]^{\beta:=\left[\begin{array}{c}b_1\cdot \\ \vdots \\ b_r\cdot \end{array}\right]} & \\
 \bigoplus_{a\in Q_1:h(a)=k} e_{t(a)} \pi(\idealM) e_\ell \ar[ur]^{\alpha:=\left[\begin{array}{ccc}a_1\cdot  & \cdots & a_s\cdot \end{array}\right]\qquad} & & \bigoplus_{b\in Q_1:t(b)=k} e_{h(b)} \pi(\idealM)e_\ell \ar[ll]^{\gamma:=\text{ {\tiny $\left[\begin{array}{ccc}
 \partial_{b_1a_1}(S) \cdot & \cdots & \partial_{b_r a_1}(S)\cdot \\
 \vdots & \ddots & \vdots\\
 \partial_{b_1a_s}(S) \cdot & \cdots & \partial_{b_ra_s}(S)\cdot
 \end{array}\right]$}}}
 }
$$
and
$$
\xymatrix{
 & e_{k}\cdot \widetilde{\pi}(\widetilde{\idealM})\cdot e_\ell \ar[dl]_{\mathfrak{b}:=\left[\begin{array}{c}a_1^*\cdot \\ \vdots \\ a_s^*\cdot \end{array}\right]\quad}& \\
 \bigoplus_{a\in Q_1:h(a)=k} e_{t(a)}\cdot \widetilde{\pi}(\widetilde{\idealM})\cdot e_\ell
 \ar[rr]_{\mathfrak{c}:=\text{ {\tiny $\left[\begin{array}{ccc}\partial_{a_1^*b_1^*}(\widetilde{\mu}_k(S)) \cdot & \cdots & \partial_{a_s^*b_1^*}(\widetilde{\mu}_k(S))\cdot \\
 \vdots & \ddots & \vdots\\
 \partial_{a_1^*b_r^*}(\widetilde{\mu}_k(S)) \cdot & \cdots & \partial_{a_s^*b_r^*}(\widetilde{\mu}_k(S))\cdot
 \end{array}\right]$}}}
 & & \bigoplus_{b\in Q_1:t(b)=k} e_{h(b)}\cdot \widetilde{\pi}(\widetilde{\idealM})\cdot e_\ell \ar[ul]_{\qquad\mathfrak{a}:=\left[\begin{array}{ccc}b_1^*\cdot  & \cdots & b_r^*\cdot\end{array}\right]} 
}
$$
satisfy $\image\mathfrak{c}\subseteq\image\beta\alpha$, $\image\gamma\subseteq\image\mathfrak{ba}$, $\ker\mathfrak{a}=\image\beta$ and $\image\mathfrak{b}\subseteq \ker\alpha$.
\end{prop}

\begin{proof}
Since  $\mathfrak{c}$ is defined as multiplying the elements of $\bigoplus_{i=1}^s e_{t(a)}\cdot \widetilde{\pi}(\widetilde{\idealM})\cdot e_{\ell}$ by the matrix formed with the elements $\partial_{a_i^*b_j^*}(\widetilde{\mu}_k(S))=[b_ja_i]$, we deduce that $\image\mathfrak{c}\subseteq\image(\beta\alpha)$ (even though some of the elements $[b_ja_i]$ themselves may fail to belong to $\image\mathfrak{c}$).

Similarly, since $\gamma$ is defined as multiplying the elements of
$\bigoplus_{j=1}^re_{h(b_j)}\pi(\idealM)e_\ell=
\bigoplus_{j=1}^re_{h(b_j)}\cdot\widetilde{\pi}(\widetilde{\idealM})\cdot e_\ell$ by the matrix formed with the elements
$\partial_{b_ja_i}(S)=\partial_{[b_ja_i]}(\widetilde{\mu}_k(S))-a_i^*\cdot b_j^*$, and since each $\partial_{[b_ja_i]}(\widetilde{\mu}_k(S))$ is zero in the Jacobian algebra of $\widetilde{\mu}_k(Q,S)$, we have $\image\gamma\subseteq\image\mathfrak{ba}$.

By Proposition \ref{prop:spanning ker a/ker gamma for radical}, every element of $\ker\mathfrak{a}$ is congruent, modulo $\image\mathfrak{c}$, to a $\bbC$-linear combination of the $r$-tuples
$(\partial_{a_1^*b_j^*}(\widetilde{\mu}_k(S)))_{j=1}^r,\ldots,(\partial_{a_{\overline{s}}^*b_j^*}(\widetilde{\mu}_k(S)))_{j=1}^r$. Noting that $(\partial_{a_1^*b_j^*}(\widetilde{\mu}_k(S)))_{j=1}^r=([b_ja_1])_{j=1}^r,\ldots,(\partial_{a_{\overline{s}}^*b_j^*}(\widetilde{\mu}_k(S)))_{j=1}^r=([b_ja_{\overline{s}}])_{j=1}^r$, we see that, although these $r$-tuples may fail to belong to $\image\mathfrak{c}$, they certainly belong to $\image\beta_k$. Since the inclusion $\image\mathfrak{c}\subseteq\image\beta\alpha$ has already been established, we deduce that $\ker\mathfrak{a}\subseteq\image\beta$.

Take now $(q_1,\ldots,q_r)\in\image\beta\subseteq \bigoplus_{j=1}^r e_{h(b_j)}\cdot\widetilde{\pi}(\widetilde{\idealM})\cdot e_{\ell}$. That is, $(q_1,\ldots,q_r)=(b_1p,\ldots,b_rp)$ for some $p\in e_k\pi(\idealM)e_\ell$. We can write $p=\sum_{i=1}^sa_ip_i$ for some $p_i\in  e_{t(a_i)}\pi(\RQ{Q})e_\ell=e_{t(a_i)}\cdot \pi(\RQ{\widetilde{\mu}_k(Q)})\cdot e_\ell$. Thus, 
$$
\left[\begin{array}{c}q_1\\
\vdots\\
q_r\end{array}\right]=\left[\begin{array}{c}\sum_{i=1}^sb_1a_ip_i\\ \vdots\\
\sum_{i=1}^sb_ra_ip_i\end{array}\right]=
\left[\begin{array}{ccc}\partial_{a_1^*b_1^*}(\widetilde{\mu}_k(S)) & \cdots & \partial_{a_s^*b_1^*}(\widetilde{\mu}_k(S))\\
\vdots & \ddots & \vdots\\
\partial_{a_1^*b_r^*}(\widetilde{\mu}_k(S)) & \cdots & \partial_{a_s^*b_r^*}(\widetilde{\mu}_k(S))
\end{array}\right]
\left[\begin{array}{c}p_1\\ \vdots \\ p_s\end{array}\right]
$$
belongs to $\ker\mathfrak{a}$ (even though some of the $p_i$ may fail to belong to $\pi(\mathfrak{m})$).

Take $(p_1,\ldots,p_s)\in\image\mathfrak{b}\subseteq\bigoplus_{i=1}^se_{t(a_i)}\cdot\widetilde{\pi}(\widetilde{\idealM})\cdot e_\ell$. That is, $(p_1,\ldots,p_s)=(a_1^*\cdot q,\ldots,a_s^*\cdot q)$ for some $q\in e_{k}\cdot\widetilde{\pi}(\widetilde{\idealM})\cdot e_\ell$. We can write $q=\sum_{j=1}^rb_j^*\cdot q_j$ for some $q_j\in e_{h(b_j)}\cdot \widetilde{\pi}(\RQ{\widetilde{\mu}_k}(Q))\cdot e_\ell=e_{h(b_j)}\pi(\RQ{Q})e_\ell$. Since the cyclic derivatives $\partial_{[b_ja_i]}(\widetilde{\mu}_k(S))$ are certainly zero in the Jacobian algebra of $\widetilde{\mu}_k(Q,S)$, we have:
\begin{align*}
\left[\begin{array}{c}p_1\\ \vdots\\ p_s\end{array}\right]&=
\left[\begin{array}{c}\sum_{j=1}^ra_1^*\cdot b_j^*\cdot q_j\\ \vdots\\ \sum_{j=1}^ra_s^*\cdot b_j^*\cdot q_j\end{array}\right]\\
&=\left(\left[\begin{array}{ccc}\partial_{[b_1a_1]}(\widetilde{\mu}_k(S)) & \cdots & \partial_{[b_ra_1]}(\widetilde{\mu}(S))\\
\vdots & \ddots & \vdots\\
\partial_{[b_1a_s]}(\widetilde{\mu}_k(S)) & \cdots & \partial_{[b_ra_s]}(\widetilde{\mu}(S))\end{array}\right]-\left[\begin{array}{ccc}\partial_{b_1a_1}(S) & \cdots & \partial_{b_ra_1}(S)\\
\vdots & \ddots & \vdots\\
\partial_{b_1a_s}(S) & \cdots & \partial_{b_ra_s}(S)\end{array}\right]\right)\left[\begin{array}{ccc}q_1\\ \vdots \\ q_r\end{array}\right]\\
&=-\left[\begin{array}{ccc}\partial_{b_1a_1}(S) & \cdots & \partial_{b_ra_1}(S)\\
\vdots & \ddots & \vdots\\
\partial_{b_1a_s}(S) & \cdots & \partial_{b_ra_s}(S)\end{array}\right]\left[\begin{array}{ccc}q_1\\ \vdots \\ q_r\end{array}\right]\in\ker\alpha.
\end{align*}
\end{proof}

Next, we study the behavior of the $\alpha$-$\beta$-$\gamma$ linear maps arising from the projectives $P_{(Q,S)}(\ell)$ and $P_{\widetilde{\mu}_k(Q,S)(\ell)}$.

\begin{prop}\label{prop:alpha-beta-gamma-maps-for-projective} The $\bbC$-linear maps 
$$
\xymatrix{
 & e_k\jacobalg{Q,S}e_\ell  \ar[dr]^{\widehat{\beta}:=\left[\begin{array}{c}b_1\cdot \\ \vdots \\ b_r\cdot \end{array}\right]} & \\
 \bigoplus_{a\in Q_1:h(a)=k} e_{t(a)} \jacobalg{Q,S} e_\ell \ar[ur]^{\widehat{\alpha}:=\left[\begin{array}{ccc}a_1\cdot  & \cdots & a_s\cdot \end{array}\right]\qquad} & & \bigoplus_{b\in Q_1:t(b)=k} e_{h(b)} \jacobalg{Q,S}e_\ell \ar[ll]^{\widehat{\gamma}:=\text{ {\tiny $\left[\begin{array}{ccc}
 \partial_{b_1a_1}(S) \cdot & \cdots & \partial_{b_r a_1}(S)\cdot \\
 \vdots & \ddots & \vdots\\
 \partial_{b_1a_s}(S) \cdot & \cdots & \partial_{b_ra_s}(S)\cdot
 \end{array}\right]$}}}
 }
$$
and
$$
\xymatrix{
 & e_{k}\cdot \jacobalg{\widetilde{\mu}_k(Q,S)}\cdot e_\ell \ar[dl]_{\widehat{\mathfrak{b}}:=\left[\begin{array}{c}a_1^*\cdot \\ \vdots \\ a_s^*\cdot \end{array}\right]\quad}& \\
 \bigoplus_{a\in Q_1:h(a)=k} e_{t(a)}\cdot \jacobalg{\widetilde{\mu}_k(Q,S)}\cdot e_\ell
 \ar[rr]_{\widehat{\mathfrak{c}}:=\text{ {\tiny $\left[\begin{array}{ccc}\partial_{a_1^*b_1^*}(\widetilde{\mu}_k(S)) \cdot & \cdots & \partial_{a_s^*b_1^*}(\widetilde{\mu}_k(S))\cdot \\
 \vdots & \ddots & \vdots\\
 \partial_{a_1^*b_r^*}(\widetilde{\mu}_k(S)) \cdot & \cdots & \partial_{a_s^*b_r^*}(\widetilde{\mu}_k(S))\cdot
 \end{array}\right]$}}}
 & & \bigoplus_{b\in Q_1:t(b)=k} e_{h(b)}\cdot \jacobalg{\widetilde{\mu}_k(Q,S)}\cdot e_\ell \ar[ul]_{\qquad\widehat{\mathfrak{a}}:=\left[\begin{array}{ccc}b_1^*\cdot  & \cdots & b_r^*\cdot\end{array}\right]} 
}
$$
satisfy
$\image\widehat{\beta}=\image\beta$, $\image\widehat{\mathfrak{b}}=\image\mathfrak{b}$, 
$\ker\widehat{\mathfrak{a}}=\ker\mathfrak{a}$,
$\ker\widehat{\alpha}=\ker\alpha=\image\widehat{\gamma}$ and $\coker\widehat{\beta}=(\bbC e_{\ell})^{\overline{r}}\oplus\coker\beta$.
\end{prop}

\begin{proof}
Notice that
\begin{center}
    \begin{tabular}{lll}
& $e_k\jacobalg{Q,S}e_\ell=e_k\pi(\idealM)e_\ell$, & $e_{k}\cdot \jacobalg{\widetilde{\mu}_k(Q,S)}\cdot e_\ell=e_k\cdot\widetilde{\pi}(\widetilde{\idealM})\cdot e_\ell$,\\
for $i=1,\ldots,\overline{s}$:& $e_{t(a_i)}\jacobalg{Q,S}e_\ell=\mathbb{C}e_{\ell}\oplus e_{t(a_i)}\pi(\idealM)e_\ell$,& $e_{t(a_i)}\cdot \jacobalg{\widetilde{\mu}_k(Q,S)}\cdot e_\ell=\mathbb{C}e_{t(a_i)}\oplus e_{t(a_i)}\cdot\widetilde{\pi}(\widetilde{\idealM})\cdot e_\ell$,\\
for $j=1,\ldots,\overline{r}$:& $e_{h(b_j)}\jacobalg{Q,S}e_\ell=\mathbb{C}e_{\ell}\oplus e_{h(b_j)}\pi(\idealM)e_\ell$,& $e_{h(b_j)}\cdot \jacobalg{\widetilde{\mu}_k(Q,S)}\cdot e_\ell=\mathbb{C}e_{h(b_j)}\oplus e_{h(b_j)}\cdot\widetilde{\pi}(\widetilde{\idealM})\cdot e_\ell$.
\end{tabular}
\end{center}
Thus, the identities $\image\widehat{\beta}=\image\beta$ and $\image\widehat{\mathfrak{b}}=\image\mathfrak{b}$ are obvious.

We obviously have $\ker\alpha\subseteq\ker\widehat{\alpha}$. Take $u\in \ker\widehat{\alpha}$. Then, in the complete path algebra $\RQ{Q}$, $\widehat{\alpha}(u)\in e_kJ(S)e_\ell$. Since $Q$ does not have $2$-cycles incident to $k$, the expression of any element of $e_kJ(S)e_\ell$ as a possibly infinite linear combination of pahts involves only paths of length at least two. Use the equality
$$
\bigoplus_{a\in Q_1:h(a)=k} e_{t(a)} \jacobalg{Q,S} e_\ell = 
(\bbC e_{\ell})^{\overline{s}}\oplus \bigoplus_{i=1}^s e_{t(a_i)}\pi(\idealM)e_{\ell}.
$$
to write $u=\sum_{i=1}^{\overline{s}}x_ie_\ell+\sum_{i=1}^{s}p_i$ with $x_1,\ldots,x_{\overline{s}}\in\bbC$ and $p_i\in e_{t(a_i)}\pi(\idealM)e_{\ell}$. Then $$0=\widehat{\alpha}(u)=\sum_{i=1}^{\overline{s}}x_i a_i+\sum_{i=1}^{s}a_ip_i.$$
in the Jacobian algebra $\jacobalg{Q,S}$. 
This implies $x_1=\cdots=x_{\overline{s}}=0$, so $u\in\ker(\alpha)$. Therefore, $\ker\widehat{\alpha}=\ker\alpha$. The proof of the equality $\ker\widehat{\mathfrak{a}}=\ker\mathfrak{a}$ is also covered by the same argument.

By \cite[Lemma 10.6]{DWZ1}, $\image\widehat{\gamma}\subseteq\ker\widehat{\alpha}$. On the other hand, from Proposition \ref{prop:spanning ker a/ker gamma for radical} one can easily deduce that $\ker\alpha\subseteq\image\widehat{\gamma}$.

Now, since $\image\widehat{\beta}=\image\beta\subseteq \bigoplus_{j=1}^re_{h(b_j)}\pi(\idealM)e_{\ell}$ and $\bigoplus_{j=1}^re_{h(b_j)}\jacobalg{Q,S}e_{\ell}=(\bbC e_\ell)^{\overline{r}}\oplus\bigoplus_{j=1}^re_{h(b_j)}\pi(\idealM)e_{\ell}$, we deduce that $\coker\widehat{\beta}=(\bbC e_{\ell})^{\overline{r}}\oplus\coker\beta.$
\end{proof}

For the next proposition, note that since $\image\widehat{\beta}\subseteq\ker\widehat{\gamma}$ (by \cite[Proposition~6.1]{DWZ1}), there is a well-defined induced map $\overline{\widehat{\gamma}}:\coker\widehat{\beta}\rightarrow \bigoplus_{a\in Q_1:h(a)=k} e_{t(a)}\jacobalg{Q,S} e_\ell$. By $p:\bigoplus_{b\in Q_1:t(b)=k} e_{h(b)}\jacobalg{Q,S}\rightarrow\coker\widehat{\beta}$ we denote the standard projection to the quotient vector space.

\begin{prop}\label{prop:alpha-beta-diagrams-for-proj-of-premut-are-isomorphic}
With the above notation,
there exists a commutative diagram of $\bbC$-vector spaces and $\bbC$-linear maps 
\begin{equation}\label{eq:crucial-commutative-alpha-beta-diagram}
\xymatrix{
 & \coker\widehat{\beta}  \ar[dd]_{\psi_k} \ar[dl]_{\overline{\widehat{\gamma}}} & \\
 \bigoplus_{a\in Q_1:h(a)=k} e_{t(a)}\jacobalg{Q,S} e_\ell \ar[dd]_{[-]}  &  &  \bigoplus_{b\in Q_1:t(b)=k} e_{h(b)}\jacobalg{Q,S} e_\ell\ar[ul]_{-p} \ar[dd]^{[-]}\\
 & e_{k}\cdot \jacobalg{\widetilde{\mu}_k(Q,S)}\cdot e_\ell \ar[dl]_{\widehat{\mathfrak{b}}:=\left[\begin{array}{c}a_1^*\cdot \\ \vdots \\ a_s^*\cdot \end{array}\right]\quad}& \\
 \bigoplus_{a\in Q_1:h(a)=k} e_{t(a)}\cdot \jacobalg{\widetilde{\mu}_k(Q,S)}\cdot e_\ell & & \bigoplus_{b\in Q_1:t(b)=k} e_{h(b)}\cdot \jacobalg{\widetilde{\mu}_k(Q,S)}\cdot e_\ell \ar[ul]_{\qquad\widehat{\mathfrak{a}}:=\left[\begin{array}{ccc}b_1^*\cdot  & \cdots & b_r^*\cdot\end{array}\right]}
}
\end{equation}
whose vertical arrows are isomorphisms, and whose left-most and right-most vertical arrows are block-diagonal matrices of linear maps.
\end{prop}

\begin{proof}
We set the left-most and right-most vertical arrows to be the linear maps from \eqref{eq:bracket-isomorphisms}.

By Propositions \ref{prop:alpha-beta-gamma-maps-for-radicals} and \ref{prop:alpha-beta-gamma-maps-for-projective}, we have $\ker\widehat{\mathfrak{a}}=\ker\mathfrak{a}=\image\beta=\image\widehat{\beta}$.
Since $k\neq\ell$, the $\bbC$-linear map $\widehat{\mathfrak{a}}$ is surjective, so there is a well-defined induced $\bbC$-vector space isomorphism 
$$
\overline{\widehat{\mathfrak{a}}}:
\coker\widehat{\beta}=
\left(\bigoplus_{b\in Q_1:t(b)=k} e_{h(b)}\cdot \jacobalg{\widetilde{\mu}_k(Q,S)}\cdot e_\ell\right)/\ker\widehat{\mathfrak{a}}\rightarrow e_k\cdot \jacobalg{\widetilde{\mu}_k(Q,S)}\cdot e_\ell.
$$
We set $\psi_k:=-\overline{\widehat{\mathfrak{a}}}$.
The commutativity of the diagram \eqref{eq:crucial-commutative-alpha-beta-diagram} is immediate.
\end{proof}

\begin{proof}[Proof of Theorem \ref{thm:mut-of-proj-is-proj}]
By Proposition \ref{prop:alpha-beta-diagrams-for-proj-of-premut-are-isomorphic}, the decorated representations $\widetilde{\mu}_k(P_{(Q,S)}(\ell),0)$ and $(P_{\widetilde{\mu}_k(Q,S)},0)$ of $\jacobalg{\widetilde{\mu}_k(Q,S)}$ are isomorphic.

By \cite[Theorem 4.6]{DWZ1}, there exist a trivial quiver with potential $(C,T)$ and a $\bbC$-algebra isomorphism $\varphi:\RQ{\mu_k(Q)\oplus C}\rightarrow\RQ{\widetilde{\mu}_k(Q)}$ acting as the identity on the set of idempotents $\{e_j\suchthat j\in Q_0\}$, such that $\varphi(\mu_k(S)+T)$ is cyclically equivalent to $\widetilde{\mu}_k(S)$. Denoting $\iota:\RQ{\mu_k(Q)}\rightarrow \RQ{\mu_k(Q)\oplus C}$ the canonical inclusion, by \cite[Propositions 3.7 and 4.5]{DWZ1} we have induced $\bbC$-algebra isomorphisms
$$
\xymatrix{
\jacobalg{\mu_k(Q,S)}\ar[r]^{\overline{\iota}\qquad\quad} & \jacobalg{\mu_k(Q)\oplus C,\mu_k(S)+T} \ar[r]^{\qquad\quad \overline{\varphi}} &\jacobalg{\widetilde{\mu}_k(Q,S)}.
}
$$
Thus, the theorem follows from \cite[Definition 10.4 and (10.22)]{DWZ1}.
\end{proof}

\begin{remark}
 The results and arguments used throughout this section suggest that it may be possible to define mutations of infinite-dimensional representations of $\jacobalg{Q,S}$. It would be interesting to pursue this possibility, and to try and find out what the corresponding notion of $F$-polynomial or $F$-series should be.
\end{remark}

\section{Duality commutes with mutation}

For a quiver with potential $(Q,S)$ we denote by $(Q^{\operatorname{op}},S^{\operatorname{op}})$ the opposite QP. For any decorated representation $\mathcal{M}=(M,\mathbf{v})$ of $(Q,S)$, the pair $\mathcal{M}^{\star}:=(M^{\star},\mathbf{v})$ is a decorated representation of $(Q^{\operatorname{op}},S^{\operatorname{op}})$, where $M^{\star}:=\Hom_{\bbC}(M,\bbC)$ is the usual vector space duality. See \cite[\S7]{DWZ}.

\begin{prop}\label{prop:duality-commutes-with-mutation}
If $k$ is a vertex not incident to any $2$-cycle of $Q$, then $\mu_{k}(Q^{\operatorname{op}},S^{\operatorname{op}})=(\mu_{k}(Q)^{\operatorname{op}},\mu_k(S)^{\operatorname{op}})$ as quivers with potential, and for any decorated representation $\mathcal{M}$ of $(Q,S)$, the decorated representations $\mu_k(\mathcal{M}^{\star})$ and $\mu_k(\mathcal{M})^{\star}$ of the QP $\mu_{k}(Q^{\operatorname{op}},S^{\operatorname{op}})$ are isomorphic.
\end{prop}

\begin{proof}
To compute $\mu_k(\mathcal{M})$ and $\mu_k(\mathcal{M}^{\star})$, one needs to consider the $\alpha$-$\beta$-$\gamma$-triangles 
$$
\xymatrix{
 & M_k \ar[dr]^{\beta} & \\
 M_{\operatorname{in}} \ar[ur]^{\alpha} & & M_{\operatorname{out}} \ar[ll]_{\gamma}
}
\quad \text{and} \quad
\xymatrix{
 & M_k^\star \ar[dl]_{\alpha^\star} & \\
 M_{\operatorname{in}}^\star \ar[rr]_{\gamma^\star} & & M_{\operatorname{out}}^\star \ar[ul]_{\beta^\star}
}
$$
Pick two sections $\sigma:\ker\alpha/\image\gamma\rightarrow\ker\alpha$, $s:\ker(\beta^\star)/\image(\gamma^\star)\rightarrow\ker(\beta^\star)$, and a retraction $\rho:M_{\operatorname{out}}\rightarrow \ker\gamma$. The map
$$
\psi_k:\coker(\alpha^\star)\oplus\frac{\ker(\beta^\star)}{\image(\gamma^\star)}\oplus \bbC^{v_k}\longrightarrow (\coker\beta)^\star\oplus\left(\frac{\ker\alpha}{\image\gamma}\right)^\star\oplus(\bbC^{v_k})^\star
$$
which for $f\in M_{\operatorname{in}}^\star$, $g\in M_{\operatorname{out}}^\star$ and $\mathbf{z}\in\bbC^{v_k}$ is given by
$$
\psi_k(f+\image(\alpha^\star),g+\image(\gamma^\star),\mathbf{z}):=(f|_{\image\gamma}\circ\overline{\gamma}+\overline{g|_{\ker\gamma}\circ\rho},f|_{\ker\alpha}\circ\sigma,\langle -,\mathbf{z}\rangle)
$$
is a well-defined isomorphism of $\bbC$-vector spaces, where $\langle \mathbf{y},\mathbf{z}\rangle:=\sum_{j=1}^{v_k}y_jz_j$ for $\mathbf{y}=(y_j)_{j=1}^{v_k},\mathbf{z}=(z_j)_{j=1}^{v_k}$. Furthermore, the diagram
$$
\xymatrix{
 & &\coker(\alpha^\star)\oplus\frac{\ker(\beta^\star)}{\image(\gamma^\star)}\oplus \bbC^{v_k} \ar[drr]^{\quad \left[\begin{array}{ccc}\overline{\gamma^\star} & i\circ s & 0\end{array}\right]} \ar[dd]_{\psi_k} & &\\
 M_{\operatorname{in}}^\star \ar[urr]^{\left[\begin{array}{c}-q\\ 0\\ 0\end{array}\right]\quad } \ar[dd]_{\myid} & & & & M_{\operatorname{out}}^\star \ar[dd]^{\myid} \\
 & & (\coker\beta)^\star\oplus\left(\frac{\ker\alpha}{\image\gamma}\right)^\star \oplus(\bbC^{v_k})^\star \ar[drr]_{\left[\begin{array}{ccc}-p^\star & 0 & 0\end{array}\right]} & &\\
 M_{\operatorname{in}}^\star \ar[urr]_{\left[\begin{array}{c}\overline{\gamma}^\star \\ (\iota\sigma)^\star \\ 0\end{array}\right]} & &  & & M_{\operatorname{out}}^\star
}
$$
commutes, where $q:M_{\operatorname{in}}^\star\rightarrow\coker(\alpha^\star)$ is the projection and $i:\ker(\beta^\star)\rightarrow M_{\operatorname{out}}^\star$ is the inclusion.

It remains to show that $\dim_{\mathbb{C}}\left(\frac{\ker\beta}{\ker\beta\cap\image\alpha}\right)=\dim_{\mathbb{C}}\left(\frac{\ker(\alpha^\star)}{\ker(\alpha^\star)\cap\image(\beta^\star)}\right)$. Now,
\begin{align*}
\ker(\alpha^\star)&=\{\varphi:M_k\rightarrow\mathbb{C}\suchthat \varphi\ \text{is} \ \mathbb{C}\text{-linear and} \ \image\alpha\subseteq\ker\varphi\}\\
& \cong \left(M_k/\image\alpha\right)^\star,\\
\image(\beta^\star)&=\{\psi\circ\beta:M_k\rightarrow\mathbb{C}\suchthat \psi\in M_{\operatorname{out}}^\star\}\\
&=\{\varphi:M_k\rightarrow\mathbb{C}\suchthat\varphi\ \text{is} \ \mathbb{C}\text{-linear and} \ \ker\beta\subseteq\ker\varphi\},\\
\text{hence} \qquad 
\ker(\alpha^\star)\cap\image(\beta^\star) &=
\{\varphi\in M_k^\star\suchthat \image\alpha+\ker\beta\subseteq\ker\varphi\}\\
&\cong \left(M_k/(\image\alpha+\ker\beta)\right)^\star.
\end{align*}
Therefore,
\begin{align*}
\dim_{\mathbb{C}}\left(\frac{\ker(\alpha^\star)}{\ker(\alpha^\star)\cap\image(\beta^\star)}\right)
&=-\dim_{\mathbb{C}}(\image\alpha)+\dim_{\mathbb{C}}(\image\alpha+\ker\beta)\\
&= \dim_{\mathbb{C}}(\ker\beta)-\dim_{\mathbb{C}}(\image\alpha\cap\ker\beta)\\
&=\dim_{\mathbb{C}}\left(\frac{\ker\beta}{\ker\beta\cap\image\alpha}\right).
\end{align*}
The proposition follows.
\end{proof}


\begin{coro}\label{coro:mut-of-inj-is-inj}
    For any quiver with potential $(Q,S)$ and any pair of distinct vertices $k,\ell\in Q_0$, with $k$ not  incident to $2$-cycles of $Q$, if the indecomposable injective $I_{(Q,S)}(\ell)$ is finite-dimensional, then so is the indecomposable injective $I_{\mu_k(Q,S)}(\ell)$, and the decorated representations $\mu_k(I_{(Q,S)}(\ell),0)$ and $(I_{\mu_k(Q,S)}(\ell),0)$ of the Jacobian algebra $\jacobalg{\mu_k(Q,S)}$ are isomorphic.
\end{coro}

\begin{proof}
Write $(M,0):=\mu_k(I_{(Q,S)}(\ell),0)$.
Suppose that $I_{(Q,S)}(\ell)$ is finite-dimensional. Then both $I_{(Q,S)}(\ell)$ and $M$ are finite-dimensional and nilpotent by \cite[Section 10]{DWZ1}. Their nilpotency indices are respectively equal to the nilpotency indices of their duals $I_{(Q,S)}(\ell)^\star$ and $M^\star$ (which are modules over the Jacobian algebras $\jacobalg{Q^{\operatorname{op}},S^{\operatorname{op}}}$ and $\jacobalg{\mu_k(Q^{\operatorname{op}},S^{\operatorname{op}}})$, respectively). Let $m\in\mathbb{Z}_{>0}$ be greater than the maximum of these nilpotency indices. 

By Proposition \ref{prop:duality-commutes-with-mutation}, $\mu_k(I_{(Q,S)}(\ell)^\star,0)$ and $(M^\star,0)$ are isomorphic. Since $I_{(Q,S)}(\ell)^\star$ is projective in the category of those representations of $\jacobalg{Q^{\operatorname{op}},S^{\operatorname{op}}}$ that have nilpotency index $\leq m$, $M^\star$ is projective in the category of those representations of $\jacobalg{\mu_k(Q^{\operatorname{op}},S^{\operatorname{op}}})$ that have nilpotency index $\leq m$. 

Thus, $M$ is injective in the category of those representations of $\jacobalg{\mu_k(Q,S)}$ that have nilpotency index $\leq m$.
By \cite[Lemma 2.2(iii) and Lemma 3.3]{CILFS}, this implies that $M$ is injective as a $\jacobalg{\mu_k(Q,S)}$-module.

Since $I_{(Q,S)}(\ell)$ is indecomposable, so is $M$. Furthermore, the simple at $\ell$ is contained in the socle of $M$. It follows that $M\cong I_{\mu_k(Q,S)}(\ell)$.
\end{proof}

	\section{Background on cluster algebras of geometric type}
	In this section we introduce notions related to cluster algebras of geometric type following \cite{Ke13} with the convention of the $B$-matrix taken from \cite{DWZ1, DWZ}.
	\subsection{Definition of the cluster algebra}\label{Def:clusteralg}
	Let $m,n\in \mathbb{Z}_{\ge 1}$, $n\le m$ and $\mathcal{F}$ be the function field ${\mathbb{Q}}(x_1,\ldots,x_m)$ in $m$ indeterminants $x_1,\ldots,x_m$.
An \emph{ice quiver $Q$ of type $(m,n)$} is a quiver with $m$ vertices which does not have any arrows between vertices $i,j$ with $n+1\le i,j$. The vertices $k\le n$ of $Q$ are called \emph{mutable vertices} and the vertices $k\ge n+1$ of $Q$ are called \emph{frozen vertices}.

The data of the quiver $Q$ can equivalently given by $m \times n$-matrix $B=B_Q=(b_{i,j})$ such that $b_{i,j}$ is defined by the difference of the number of arrows from $j$ to $i$ and the number of arrows from $i$ to $j$. Note that the $n\times n$-submatrix $\tilde{B}$ given by the first $n$ rows is skew-symmetric. We call $\tilde{B}$ the \emph{principal part} of $B$. 

The mutation rule for quivers is expressed in terms of skew-symmetric matrices as follows. If $B_Q=(b_{i,j})$ and $B_{\mu_k(Q)}=(b'_{i,j})$, then
\begin{equation*}
b'_{i,j} = \begin{cases} -b_{i,j} & \text{if } k\in \{i,j\}\\
b_{i,j}+[b_{i,k}]_+[b_{k,j}]_{+} - [-b_{i,k}]_+[-b_{k,j}]_+ & \text{else,} \end{cases}
\end{equation*}

where for $a \in \mathbb{Z}$ the number $[a]_+$ is defined as the maximum of $a$ and $0$.

An \emph{$\mathcal{A}$-cluster seed} is a pair $(Q,u)$ where ${Q}$ is an ice quiver of type $(m,n)$ and $u=(u_1,\ldots,u_m)$ is a sequence of elements of $\mathcal{F}$ which freely generate $\mathcal{F}$.

For a seed $({Q},u)$ and $k\in Q_0$, $k\le n$, the \emph{mutation} $\mu_k ({Q},u)$ of the seed at $k$ is the seed $(\mu_k (Q),u')$, where $\mu_k({Q})$ is defined in Definition \ref{def:threesteps} and $u'$ is obtained from $u$ by replacing $u_k$ by the element $u'_k\in \mathcal{F}$ defined by
\begin{equation}
    u'_k u_k=\displaystyle\prod_{h(\alpha)=k}u_{t(\alpha)}+\displaystyle\prod_{t(\alpha)=k}u_{h(\alpha)},
\end{equation}
where $\alpha$ ranges over all arrows of $Q$ incident to $k$.

The \emph{initial seed} of ${Q}$ is $({Q},\{x_1,\ldots,x_m\})$. A \emph{cluster} associated to ${Q}$ is a sequence $u$ which appears in a seed $({Q}',u')$ obtained from the initial seed by iterated mutation. The elements of the clusters are called \emph{$\mathcal{A}$-cluster variables}. 

The cluster algebra $\mathcal{A}_{{Q}}$ is the $\mathbb{Q}-$subalgebra of $\mathcal{F}$ generated by all $\mathcal{A}$-cluster variables obtained from the initial seed $({Q},\{x_1,\ldots,x_m\})$ by iterated mutations.

\subsection{Parametrization of seeds by the regular $n$-tree} We have a convenient parametrization of the seeds which can be obtained from the initial seed by iterated mutation. Let $\mathbb{T}_n$ be the $n$-regular tree whose edges are labeled by the integers $1,\ldots,n$ such that each vertex $k$ has $n$ edges incident to $k$ which each carry a different label. 

Let $t_0$ be a vertex of $\mathbb{T}_n$. To each vertex $t$ of $\mathbb{T}_n$ we associate a seed $({Q}(t),u(t))$ such that $({Q}(t_0),u(t_0))$ is the initial seed and whenever a vertex $t$ is connected to a vertex $t'$ by an edge with label $k$, the seeds $({Q}(t),u(t))$ and $({Q}(t'),u(t'))$ are related by a mutation at $k$. We write $x_i(t)$, $1\le i \le n$ for the $\mathcal{A}$-cluster variables, in the seed $(Q(t),u(t))$. We also write $B(t)$ for the matrix $B_{Q(t)}.$

\subsection{$y$-seeds}
To every $t\in \mathbb{T}_n$ we associate another sequence $Y(t)=(y_1(t),\cdots, y_m(t))$ of elements in $\mathcal{F}$ called the $y$-variables. The triple $({Q}(t),u(t),Y(t))$ is called a \emph{seed pattern} of $\mathcal{A}_{{Q}}$. 
For the initial seed $t_0$ we set $Y(t_0)=(y_1,\ldots,y_m)$. If $t$ and $t'$ are connected in $\mathbb{T}_n$ by an edge with label $k$ and $B_{Q(t)}=(b_{i,j})$, we set
\begin{equation}\label{y-mut}
y_j(t')=\begin{cases} y_k(t)^{-1} & \text{ if }j=k, \\
y_j(t)(1+y_k(t)^{-\text{sgn}(b_{j,k})})^{-b_{j,k}} & \text{ if }j\ne k.
\end{cases}
\end{equation}
Note that this definition differs from the one given in \cite{Ke13} due to our different convention of the $B$-matrix.

Note also that, in difference to the frozen $\mathcal{A}$-cluster variables, in general $y_{\ell}(t)\ne y_{\ell}(t')$ for $\ell> n$ and distinct $t,t'\in \mathbb{T}_n$. If we want to stress that we consider the cluster algebra with initial seed pattern $({Q},\{x_1,\ldots,x_m\},(y_1,\ldots,y_m))$ at $t_0$, we write $x_j(t)=x_j^{Q,t_0}(t)$, $y_j(t)=y_j^{Q,t_0}(t)$.

$$\xymatrix{
t \ar@{-}[r]^{k} & t'
}$$

\section{Behavior of $F$-polynomials of injective QP-representations}
In this section we assume that $Q$ has no frozen vertices. 

\subsection{$\mathbf{g}$-vectors and $\mathbf{h}$-vectors}

Let $(Q,S)$ be a non-degenerate QP.
Recall from \cite[(1.13)]{DWZ1}, that the \emph{injective $\mathbf{g}$-vector} of a decorated representation $\mathcal{M}=(M,\mathbf{v})$ of $(Q,S)$ is the vector $\mathbf{g}^{(Q,S)}_{\mathcal{M}}=(g_j^{(Q,S)}(\mathcal{M}))_{j\in Q_0}$ defined by
$$
g_j^{(Q,S)}(\mathcal{M}):=\dim_{\bbC}\ker\gamma_j-\dim_{\bbC}M_j+v_j.
$$
By \cite[(10.4)]{DWZ} (see also \cite[Lemma 4.7, Proposition 4.8]{P11} and \cite[Lemmas 3.1, 3,3 and 3.4]{CILFS}), can be computed in terms of minimal injective presentations as follows. If 
$$
0\rightarrow M\rightarrow \bigoplus_{j\in Q_0}I_{(Q,S)}(j)^{a_j}\rightarrow \bigoplus I_{(Q,S)}(j)^{b_j}
$$
is a minimal injective presentation of $M$, then for every $j\in Q_0$ we have 
\begin{equation}\label{eq:g-vector-via-inj-presentations}
g_j^{(Q,S)}(\mathcal{M})=b_j-a_j+v_j.
\end{equation}

Recall also from \cite[(3.2)]{DWZ}, that the \emph{$\mathbf{h}$-vector} $\mathbf{h}_{\mathcal{M}}^{(Q,S)}=(h_j^{(Q,S)}(\mathcal{M}))_{j\in Q_0}$ is defined by
$$
h_j^{(Q,S)}(\mathcal{M}):=-\dim_{\bbC}\ker\beta_j.
$$
According to \cite[Proof of Lemma 5.2]{DWZ}, 
\begin{equation}\label{eq:h-vectors-under-mut}
    h_k^{(Q,S)}(\mathcal{M})-h_k^{\mu_k(Q,S)}(\mu_k(\mathcal{M}))=g_k^{(Q,S)}(\mathcal{M})
\end{equation}

\begin{lemma}\label{lemma:h-vectors-of-injectives}
If $(Q,S)$ is a non-degenerate QP, $k$ and $\ell$ are two distinct vertices of $Q$, and the injective $\jacobalg{(Q,S)}$-module $I_{(Q,S)}(\ell)$ is finite-dimensional, then $h_k^{(Q,S)}(I_{(Q,S)}(\ell),0)=0$.
\end{lemma}

\begin{proof} By \cite[Corollary 10.9]{DWZ} (see also \cite[Proposition 3.5]{CILFS}), the $E$-invariant of $(I_{(Q,S)}(\ell),0)$ is equal to $0$. Consequently, by \cite[(4.22) and Corollary 8.3]{DWZ}, one of the two numbers $h_k^{(Q,S)}(I_{(Q,S)}(\ell),0)$ and $h_k^{\mu_k(Q,S)}(\mu_k(I_{(Q,S)}(\ell),0)$ must be equal to $0$.
On the other hand, $g_k^{(Q,S)}(I_{(Q,S)}(\ell),0)=0$ by \eqref{eq:g-vector-via-inj-presentations}. The lemma thus follows from \eqref{eq:h-vectors-under-mut}.
\end{proof}

\subsection{F-Polynomials of QP-representations}\label{Sec:Fpolrep}

Using the notation of Section \ref{subsec:quiver-mutations} let $(Q,S)$ be a non-degenerate quiver with potential and let $\mathcal{M}=(M,V)$ be a QP-representation of $(Q,S)$. 

For $\ee\in \mathbb{Z}_{\ge 0}^{|Q_0|}$, we denote by

$$\Gr_{\ee}(\mathcal{M})=\{\mathcal{N}\subset \mathcal{M} \mid \mathcal{N} \text{ is a QP-subrepresentation of }\mathcal{M} \text{ of dimension vector }\ee\}.$$

We define the $F$-polynomial of $\mathcal{M}$ by
\begin{equation*}
F_{\mathcal{M}}(u_1,\ldots,u_n)=\sum_{\ee \in \mathbb{Z}_{\ge 0}^n} \chi(\Gr_{\ee})(\mathcal{M})\displaystyle\prod_{i=1}^{n}u_i^{e_i},
\end{equation*}
where $\chi$ denotes the topological Euler-characteristic.

The following theorem shows that $F-$polynomial of the injective representation to the vertex $\ell$ is $y$-seed mutation invariant provided we do not mutate at $\ell$.

\begin{thm}\label{profFpol} Let $\xymatrix{
t \ar@{-}[r]^{k} & t'}$ be an edge in $\mathbb{T}_n$ with $k\ne \ell$. Then 
$$F_{I_{(Q,S)}(\ell)}(\mu_k(y_1),\ldots,\mu_k(y_n))=F_{\mathcal{I}_{(Q',S')}(\ell)}(y_1,\ldots,y_n),$$
where $(Q',S')$ is obtained from $(Q,S)$ by mutation at $k$.
\end{thm}

\begin{proof} This follows from a combination of Corollary \ref{coro:mut-of-inj-is-inj}, Lemma \ref{lemma:h-vectors-of-injectives} and \cite[Lemma 5.2]{DWZ}.
\end{proof}

\section{Cluster varieties and computation of Landau-Ginzburg potentials}

\subsection{Cluster varieties}
Another viewpoint due to Fock-Goncharov \cite{FG} on seed patterns is that they define algebraic tori which glue to positive spaces, called cluster varieties, via cluster mutations. We recall this construction in the following.

Fix an ice quiver $Q=Q(t_0)$ of type $(m,n)$. To each $t\in \mathbb{T}_n$ we assign a collection of $\mathcal{A}$-cluster variables $\{A_j(t)\}_{j\in [1,m]}$ and $\mathcal{X}$-cluster variables $\{X_j(t)\}_{j \in [1,m]}$ and tori
$$\mathcal{A}_{t}:=\operatorname{Spec}\mathbb{C}[A_j(t)^{\pm 1} \mid j\in [1,m]], \qquad {{\mathcal{X}}}_{t}:=\operatorname{Spec}\mathbb{C}[X_j(t)^{\pm 1} \mid j\in [1,m]],$$
called the \emph{$\mathcal{A}$-cluster torus} and the \emph{$\mathcal{X}$-cluster torus} associated to $t$, respectively. We call $\{A_j(t)\}_{j \ge n+1}$ and $\{X_j(t)\}_{j \ge n+1}$ the frozen $\mathcal{A}-$ and $\mathcal{X}$-cluster variables respectively.

Assume that the quiver $Q(t')$ is obtained from the quiver $Q(t)$ by mutation at the vertex $k\le n$ and $B(t)=(b_{i,j})$ is given as in Section \ref{Def:clusteralg}. We define birational transition maps (see \cite[Equations (13) and (14)]{FG})
\begin{align} \notag
{\mu_k^*} A_i(t') &   = \begin{cases}   {A_k(t)^{-1}  \left(\displaystyle\prod_{j \,:\, b_{k,j} >0} A_j(t) ^{b_{k,j}}   + \displaystyle\prod_{j \,:\, b_{k,j} < 0} A_j(t) ^{-b_{k,j}}\right)} &   \text{if $i=k$,} \\   A_i(t)  &   \text{else,} \end{cases} \\ \label{X-mutation}
 \check{\mu}_k^* X_i(t')  &   = \begin{cases}   X_k(t) ^{-1} &   \text{if $i=k$,} \\   X_i(t)  \left(1+X_k(t) ^{-\text{sgn} (b_{k,i})}\right)^{-b_{k,i}} &   \text{else,} \end{cases}
\end{align}
which we call \emph{$\mathcal{A}$- and $\mathcal{X}$-cluster mutation}, respectively. 

\begin{remark}
Note that our definition slightly differs from \cite{FG,GHKK} due to our different convention of the $B$-matrix in Section \ref{Def:clusteralg}.
\end{remark}

\begin{defi} Given a fixed initial ice quiver $Q=Q(t_0)$ for a $t_0\in \mathbb{T}_n$, the initial collection  of $\mathcal{A}$-cluster variables $\{A_k(t_0)\}_{k\in [1,m]}$ and $\mathcal{X}$-cluster variables $\{X_k(t_0)\}_{k \in [1,m]}$ and the algebraic tori $\mathcal{A}_{t_0}=\operatorname{Spec}\mathbb{C}[A_k(t_0)^{\pm 1}\mid k\in [1,m]],$ ${{\mathcal{X}}}_{t_0}=\operatorname{Spec}\mathbb{C}[X_k(t_0)^{\pm 1} \mid k\in [1,m]]$. The $\mathcal{A}$- and $\mathcal{X}$-cluster variety, respectively, is defined as the scheme 
$$\mathcal{A}:=\displaystyle\bigcup_{t' \in \mathbb{T}_n} \mathcal{A}_{t} \qquad \mathcal{X}=\displaystyle\bigcup_{t'\in \mathbb{T}_n} \mathcal{X}_{t}, $$
obtained by gluing the tori $\mathcal{A}_{t}$ and $\mathcal{X}_{t}$ along the $\mathcal{A}$- and $\mathcal{X}$-cluster mutation, respectively.
\end{defi}

\subsection{Partial compactification of cluster varieties and their Landau-Ginzburg potentials}\label{partcomp}

Many interesting spaces arise as so-called partial compactification of cluster varieties \cite{GHKK}. The partial compactification $\overline{\mathcal{A}}$ of $\mathcal{A}$ is hereby defined by adding the divisors corresponding to the vanishing locus of the frozen cluster variables.

In \cite{GHKK} a Landau-Ginzburg potential $W$ on $\mathcal{X}$ is defined as the sum of certain global monomials attached to the frozen cluster variables. We only give the definition in the case that every frozen cluster variable has an optimized seed.

\begin{defi}\label{potdefi}  Let $t,t_0\in \mathbb{T}_n$
\begin{itemize}
    \item[(i)]  Let $\ell$ be a frozen vertex of $Q(t)$. We say that $t$ is optimized for $\ell$ if $\ell$ is a sink in $Q(t)$.

\item[(ii)] Let $\mathcal{X}$ be the cluster variety with initial torus ${{\mathcal{X}}}_{t_0}=\operatorname{Spec}\mathbb{C}[X_j(t_0)^{\pm 1} \mid j\in [1,m]]$ corresponding to $Q(t_0)$. If for every $\ell\in \{n+1,\ldots,m\}$ there exists a $t_{\ell}\in \mathbb{T}_n$ which is optimized for $\ell$, we define 
$$W=\displaystyle\sum_{\ell \in \{n+1,\ldots,m\}} W_{\ell} \in \mathbb{C}[{\mathcal{X}}]$$
by 
$$\restr{W_{\ell}}{{{\mathcal{X}}}_{t_{\ell}}}=X_{\ell}(t_{\ell})^{-1}.$$
\end{itemize}
\end{defi}

Every summand $W_{\ell}$ may be obtained from a smaller quiver having $\ell$ as the only frozen vertex as the next lemma shows.

\begin{lemma}\label{subquiver}
Let $Q^{\ell}(t)$ be the full subquiver of $Q(t)$ spanned by all mutable vertices and $\ell$. Let $\mathcal{X'}$ be the cluster variety with initial torus $\mathcal{X}'_t=\operatorname{Spec}\mathbb{C}[X_j(t)^{\pm 1} \mid j\in [1,n]\cup \{\ell\}]$ corresponding to $Q^{\ell}(t)$. Identifying $\mathcal{X}'_{t}$ with a subtorus of ${{\mathcal{X}}}_{t}$, we have
$$\restr{W_{\ell}}{{\mathcal{X}}_{t}}=\restr{W_{\ell}}{{\mathcal{X}}'_{t}}.$$
Moreover
$$\restr{W_{\ell}}{{{\mathcal{X}'}}_{t}}\in \mathbb{C}[X_j(t)^{-1} \mid j\in [1,n]\cup \{X_\ell\}]$$
\end{lemma}

\begin{proof}
The fact, that 
$$\restr{W_{\ell}}{{{\mathcal{X}}}_{t}}\in \mathbb{C}[X_j(t)^{-1} \mid j\in [1,m]\}$$
was shown in \cite[Proposition 4]{KS}.

Note that by the definition of the $\mathcal{X}$-cluster mutation given in \eqref{X-mutation}, the variable $X_j(t)$ may only contribute to any monomial of $\restr{W_{\ell}}{{{\mathcal{X}}}_{t}}$ with non-zero exponent if there is a path from $t$ to $t_{\ell}$ containing an edge with label $j$, where $t_{\ell}$ is optimized for $\ell$.
Therefore $\restr{W_{\ell}}{{{\mathcal{X}'}}_{t}}=\restr{W_{\ell}}{{{\mathcal{X}}}_{t}}\in \mathbb{C}[X_j(t)^{-1} \mid j\in 1,n]\cup \{X_{\ell}\}]$ and, seen as an element of $\mathbb{C}[X_j(t)^{-1} \mid j\in [1,m]$, $\restr{W_{\ell}}{{{\mathcal{X}}}_{t}}$ is independent of adding or deleting frozen vertices $\ell'\ne \ell$.
\end{proof}

\subsection{Landau-Ginzburg potentials as sums of $F$-polynomials}

Using the notation of Section \ref{backgroundQP} let $(Q,S)$ be a non-degenerate quiver with potential and let $\mathcal{M}=(M,\mathbf{v})$ be a decorated representation of $(Q,S)$. Recall that this means by definition that $M$ is finite-dimensional. 

Let 
$$\Gr^{\ee}(\mathcal{M})=\{\mathcal{N} \mid \mathcal{N}=\mathcal{M}/\mathcal{N'} \text{ for a QP-suprepresentation } \mathcal{N'} \text{ of }\mathcal{M} \text{ and the dimension vector of }\mathcal{N} \text{ is }\ee\}.$$

We define the dual $F$-polynomial of $\mathcal{M}$ by

\begin{equation*}
F^{\vee}_{\mathcal{M}}(u_1,\ldots,u_m)=\sum_{\ee \in \mathbb{Z}_{\ge 0}^m} \chi(\Gr^{\ee}(\mathcal{M}))\displaystyle\prod_{i=1}^{m}u_i^{e_i}.
\end{equation*}

In our expression of the Landau-Ginzburg potential we use the dual $F$- poylynomial of the indecomposable projective $(Q^{\ell},S^{\ell})$-representation ${P}_{(Q^{\ell},S^{\ell})}(\ell)$ to the vertex $\ell$ of the quiver $Q^{\ell}$ as defined in Lemma \ref{subquiver}, where $S^{\ell}$ is the restriction of the potential $S$ on $Q$ to $Q^{\ell}$ defined as follows.

\begin{defilem}[{\cite[Definition 8.8]{DWZ1}, \cite[Corollary 22]{LF09}}]
Let $(Q,S)$ be a QP and $\tilde{Q}$ be a full subquiver of $Q$. Let $\Psi_{\tilde{Q}}:\RQ{Q} \rightarrow \RQ{\tilde{Q}}$ be the $\mathbb{C}$-algebra homomorphism such that for $a\in Q$ we have $\Psi_{\tilde{Q}}(a)=\begin{cases} a & \text{if }a\in \tilde{Q} \\ 0 & \text{if }a\notin \tilde{Q} \end{cases}$. We define the restriction $\restr{S}{\tilde{Q}}$ of $S$ on $\tilde{Q}$ as $\restr{S}{\tilde{Q}}:=\Psi_{\tilde{Q}}(S)$. If $(Q,S)$ is non-degerenate then $(\tilde{Q}, \restr{S}{\tilde{Q}})$ is non-degenerate.
\end{defilem}

We get an expression of the Landau-Ginzburg potential in any $\mathcal{X}$-cluster torus as sum of $F$-polynomials (without constant term) of projective QP-representations in the following theorem.

\begin{thm}\label{potasfpol} Let $\mathcal{X}$ be the cluster variety with initial torus ${{\mathcal{X}}}_{t}=\operatorname{Spec}\mathbb{C}[X_k(t)^{\pm 1} \mid k\in [1,m]]$ and initial quiver $Q$. Assume that for every $\ell\in \{n+1,\ldots,m\}$ there exists a $t_{\ell}\in \mathbb{T}_n$ which is optimized for $\ell$. For any $\ell \in \{n+1,\ldots,m\}$ let $P(\ell):={P}_{(Q^{\ell},S^{\ell})}(\ell)$ be the indecomposable projective $\jacobalg{Q^{\ell},S^{\ell}}$-representation of  to the vertex $\ell$. Then $(P(\ell),0)$ is a decorated representation of $(Q^{\ell},S^{\ell})$, i.e. $P(\ell)$ is finite dimensional.

Moreover
\begin{equation}\label{mainthmeq}
\restr{W_{\ell}}{{{\mathcal{X}}}_{t}}=F^{\vee}_{{P}_{(Q^{\ell},S^{\ell})}(\ell)}(X_1(t)^{-1},\ldots X_m(t)^{-1}-1=X_{\ell}(t)^{-1}F^{\vee}_{\text{rad}_{(Q^{\ell},S^{^\ell})}{P_{(Q^{\ell},S^{\ell})}(\ell)}}(X_1(t)^{-1},\ldots X_m(t)^{-1}),\end{equation}
where $\text{rad}_{(Q^{\ell},S)}{P}_{(Q^{\ell},S^{\ell)}}(\ell)$ is the radical of $P(\ell)$, $S$ is a potential such that $(Q,S)$ is a non-degenerate QP and $S^{\ell}=\restr{S}{Q^{\ell}}$.
\end{thm}

\begin{proof}
Let $t=t_{\ell}$. Then $\ell$ is a sink in $Q$ and hence ${P}_{(Q^{\ell},S^{\ell})}(\ell)$ is the simple QP-module to the vertex $\ell$. Since mutation preserves the property of being a QP-module by construction (see \cite[Section 10, Equation 10.6, Equation 10.7, Proposition 10.7]{DWZ1}) the fact that $(P(\ell),0)$ is a QP-representation of $(Q^{\ell},S^{\ell})$ follows from Theorem \ref{thm:mut-of-proj-is-proj}.

We show the first equality of \eqref{mainthmeq} by induction. For $t=t_{\ell}$ we have
$$F^{\vee}_{{P}_{(Q^{\ell},S)}(\ell)}(u_1,\ldots,u_m)=u_{\ell}.$$
We hence get by Definition \ref{potdefi},
$$F^{\vee}_{{P}_{(Q^{\ell},S)}(\ell)}(X_1(t)^{-1},\ldots X_m(t)^{-1})=X_{\ell}(t)^{-1}+1=\restr{W_{\ell}}{{{\mathcal{X}}}}_{t}+1.$$

Let us now assume the statement holds for an arbitraty $t\in \mathbb{T}_n$ and let
$$\xymatrix{
t \ar@{-}[r]^{k} & t'
}$$
be an edge in $\mathbb{T}_n$. Since 
$$\restr{W_{\ell}}{{{\mathcal{X}}}}_{t'}=\restr{W_{\ell}}{{{\mathcal{X}}}}_{t}\circ \check{\mu}_k^*,$$

it remains to show
\begin{equation}\label{eq:left}
    F^{\vee}_{{P}_{(Q^{\ell},S^{\ell})}(\ell)}(X_1(t)^{-1},\ldots X(t)_m^{-1})\circ  \check{\mu}_k^*=F^{\vee}_{{P}_{(Q'^{\ell},S'^{\ell})}(\ell)}(X_1(t)^{-1},\ldots X_m(t)^{-1}),
\end{equation}
where $(Q'^{\ell},S'^{\ell})$ is obtained from $(Q^{\ell},S^{\ell})$ by mutation at $k$.

Let $(Q^{\ell})^{\operatorname{op}}$ be the opposite quiver to $Q^{\ell}$, $(S^{\ell})^{\operatorname{op}}$ the opposite potential and let ${I}_{((Q^{\ell})^{\operatorname{op}},(S^{\ell})^{\operatorname{op}})}(\ell)$ be the indecomposable injective QP-module of $((Q^{\ell})^{\operatorname{op}},(S^{\ell})^{\operatorname{op}})$ to the vertex $\ell$.

Comparing the $\mathcal{X}$-cluster mutation rule (\eqref{X-mutation}) with the $y-$seed mutation rule (\eqref{y-mut}), we note that \eqref{eq:left} is equivalent the following statement. 
\begin{equation*}
    F_{{I}_{((Q^{\ell})^{\operatorname{op}},(S^{\ell})^{\operatorname{op}})}(\ell)}(\mu_k(y_1(t)),\ldots, \mu_k(y_m(t))) =F_{{I}_{((Q'^{\ell})^{\operatorname{op}},(S'^{\ell})^{\operatorname{op}})}(\ell)}(y_1(t),\ldots y_m(t))
\end{equation*}
which is proved in Theorem \ref{profFpol}.
The last equality of \eqref{mainthmeq} follows directly from the definition of the radical. 
\end{proof}

\begin{ex}\label{ex:SL3}
Let $Q(t)$ be the following quiver:

$$
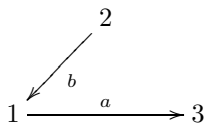
\xymatrix{
& 2 \ar[dl]^{b}  \\
1 \ar[rr]^{a} && 3 
}$$
with $\{2,3\}$ the set of frozen vertices. Since the quiver is acyclic, every potential is trivial. We have $Q(t)^2=1 \xleftarrow{b} 2$, the indecomposable projective $P(2):=P_{(Q(t)^2,0)}(2)$ to the vertex $2$ is of dimension vector $(1,1)$ and $S(2)$, the simple representation to the vertex $2$, is the only non-trivial proper quotient of $P(2)$ and has dimension vector $(0,1)$. Hence, by Theorem \ref{potasfpol}, we have
$$\restr{W_{2}}{{{\mathcal{X}}}}_{\ii}=X(t)^{-1}_2+X(t)^{-1}_2X(t)^{-1}_1$$
since both quiver Grassmannians are points.

For the other frozen vertex, we consider $Q(t)^3=1 \xrightarrow{1} 3$ and notice that in this case $P(3)$ is the simple module to the vertex $3$. Hence, by Theorem \ref{potasfpol},
$$\restr{W_{3}}{{{\mathcal{X}}}}_{\ii}=X^{-1}_3.$$
\end{ex}
\begin{remark}
Example \ref{ex:SL3} shows the need to consider the full subquiver  $Q^{\ell}$ of $Q$ since the indecomposable projectives to the vertex $2$ differ for $Q$ and $Q^{\ell}$. One way of overcoming this technically might be to work out Section \ref{Sec:Fpolrep} for ice quivers with potential following \cite{P20} or \cite{GLS20}.
\end{remark}

\section{The cluster algebra $\mathbb{C}[U]$ and applications to string cones}

\subsection{The cluster algebra $\mathbb{C}[U]$}

Let $G$ be a connected, simply-connected, simple, simply-laced algebraic group with Borel subgroup $B\subset G$. Let $B^-$ be the opposite Borel. Let $W$ be the Weyl group of $G$ with set of generators $\{s_i\mid i\in I\}$ and longest element $w_0$, where $I$ is the index set of the Dynkin diagram of $G$. For $u,v \in W$, the subset $G^{u,v}=BuB\cap B^-vB^-$ of $G$ is called the double Bruhat cell corresponding to $u,v$ and $G$ is the disjoint union of all double Bruhat cells. In \cite{BFZ2} a cluster structure on each double Bruhat cell $G^{u,v}$ is constructed. In this section we focus on the open reduced double Bruhat cell $L^{e,w_0}=U\cap B^-w_0B^-$ in the unipotent radical $U$ of $B$.

We call $\ii=(i_1,\ldots,i_N)$, $i_s\in I$ for all $s$, a reduced word for $w_0$ if $w_0=s_{i_1}\ldots s_{i_N}$ and this is a decomposition of $w_0$ into a product of generators with a minimal numbers of factors. Recall that each such decomposition has the same length which we denote by $N$. The set of all reduced words for $w_0$ is denoted by $R(w_0)$.

A choice of initial quiver to the cluster structure on $L^{e,w_0}$ is given for any $\ii \in R(w_0)$ as follows.

\begin{defi}\label{clustergraph} We associate to every reduced word $\mathbf{i} =(i_1,i_2,\ldots,i_N)\in R(w_0)$ a quiver $\Gamma_{\ii}$ with vertex set $\{j \mid j \in [1,N]\}$. For an index $j\in N$ we denote by $j^+$ the smallest index $s \in [1,N]$ such that $j<s$ and $i_{j}=i_s$. If no such $s$ exists, we set $j^+=N+1$. The vertices $j,s$, $j<s$ are connected by an edge in $\Gamma_{\ii}$ if and only if $\{j^+,s^+\}\cap [1,N]\ne \emptyset$ and one of the two conditions are satisfied
\begin{itemize}
\item[(i)] $s=j^+$,
\item[(ii)] $s < j^+ <s^+$.
\end{itemize}
An edge of type $(i)$ is directed from $j$ to $s$ and an edge of type $(ii)$ is directed from $s$ to $j$. The set of frozen vertices of $\Gamma_{\ii}$ is given by all $\ell$ such that $\ell^+=N+1$.
\end{defi}

We denote by $\mathcal{X}_{\ii}=\operatorname{Spec}\mathbb{C}[X_j^{\pm 1}\mid j \in [1,N]]$ the initial torus of the $\mathcal{X}$-cluster variety associated to $Q=\Gamma_{\ii}$. Note that for every index $i\in I$ there is a frozen vertex $\ell_i$ of $\Gamma_{\ii}$ where $\ell_i=\max \{j \in [1,N] \mid i_j=i\}$ and every frozen vertex is of that form. Thus the Landau-Ginzburg potential $W$ takes the form
$$W=\sum_{i\in I} W_{\ell_i}.$$

The quiver $\Gamma_{\ii}$ omits a non-degenerate potential by \cite{K19} which we recall in the following. We call a cycle of $\Gamma_{\ii}$ which is not divided by an edge a \emph{face of $\Gamma_{\ii}$}.

\begin{prop}\cite[Theorem 1.2.]{K19}\label{nondegred} The potential $S(\ii)$ on $\Gamma_{\ii}$ given by
$$S(\ii)=\sum\text{clockwise oriented faces}-\sum\text{anti-clockwise oriented faces}$$
is non-degenerate.
\end{prop}

\begin{ex} Let $G=\text{SL}_5(\mathbb{C})$ and $\ii=(1,2,1,3,2,1,4,3,2,1)\in R(w_0)$. The quiver $\Gamma_{\ii}$ looks as follows:

$$
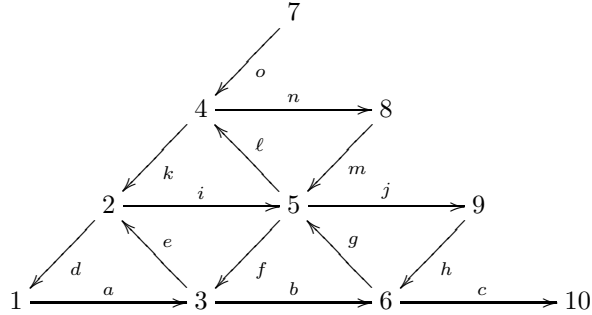
\xymatrix{
& & & 7 \ar[dl]^{o} \\
& & 4 \ar[rr]^{n} \ar[dl]^{k} & & 8 \ar[dl]^{m} \\
& 2 \ar[rr]^{i} \ar[dl]^{d} && 5 \ar[rr]^{j}  \ar[ul]_{\ell} \ar[dl]^{f} && 9 \ar[dl]^{h} \\
1 \ar[rr]^{a} && 3 \ar[rr]^{b}  \ar[ul]_{e} && 6 \ar[rr]^{c} \ar[ul]_{g} && 10
}$$
with $\{7,8,9,10\}$ the set of frozen vertices. The non-degenerate potential from Proposition \ref{nondegred} is given by
$$S(\ii)= jgh+ief+n\ell m-bfg-ade-ike.$$
\end{ex}

\subsection{Tropicalization}

We recall the notion of tropicalization from \cite{GHKK}. Let $T=(\mathbb{C}^{*})^s$ be an algebraic torus. We denote by $[T]_{trop}=\Hom(\mathbb{C}^*,T)=\mathbb{Z}^s$ its cocharacter lattice. A positive (i.e. subtraction-free) rational map $f$ on $T$, $f(x)=\frac{\sum_{u\in I} a_u x^u}{\sum_{u\in J} b_u x^u}$ with {$a_u, b_u \in \mathbb{R_{+}}$}, gives rise to a piecewise-linear map
\begin{equation*}
[f]_{trop} : [{T}]_{trop} \rightarrow [\mathbb{C}^*]_{trop}=\mathbb{Z}, \quad x\mapsto \min_{u \in I} \left\langle x,u\right\rangle - \min_{u\in J} \left\langle x,u\right\rangle,
\end{equation*}
where $\langle \cdot ,\cdot \rangle$ is the standard inner product of $\mathbb{Z}^s$. We call $[f]_{trop}$ the \emph{tropicalization} of $f$. 

For a positive rational map
$$f=(f_1, \dots, f_s) : (\mathbb{C}^*)^q \dashrightarrow (\mathbb{C}^*)^{s}$$
we define its tropicalization as 
$$[f]_{trop}:=([f_1]_{trop}, \ldots ,[f_{s}]_{trop}): [(\mathbb{C}^*)^q]_{trop} \rightarrow [(\mathbb{C}^*)^{s}]_{trop}.$$

\subsection{String cones}

Let $B(\infty)$ be the crystal basis of $U_q^-$ in the sense of \cite{Ka94} with partially inverse Kashiwara operators $\tilde{e}_i$, $\tilde{f}_i$ for $i\in I$. For each reduced word $\mathbf{i}=(i_1,i_2,\ldots,i_N)\in R(w_0)$ we define the $\ii$-string datum of an element $x\in B(\infty)$ as follows.

\begin{defi} 
An \emph{$\mathbf i$-string datum} $\text{str}_{\ii}(x)$ of $x \in B(\infty)$ is defined as a tuple $(x_1,x_2,\ldots,x_N)\in \mathbb{Z}_{\ge 0}^{N}$ determined inductively by

\begin{align*}
x_1 &= \displaystyle\max_{ k \in \mathbb{Z}_{\ge 0}} \{\tilde{e}_{i_1}^k x \in B(\infty)\}, \\
x_2 &= \displaystyle\max_{ k \in \mathbb{Z}_{\ge 0}} \{ \tilde{e}_{i_2}^{k}\tilde{e}_{i_1}^{x_1} x \in B(\infty)\}, \\
& \vdots \\
x_N &= \displaystyle\max_{ k \in \mathbb{Z}_{\ge 0}} \{ \tilde{e}_{i_N}^k \tilde{e}_{i_{N-1}}^{x_{N-1}}\cdots \tilde{e}_{i_1}^{x_1} x \in B(\infty) \}.
\end{align*}

By \cite{BZ,Lit} the subset 
$$\mathcal{C}_{\mathbf{i}}:=\{\text{str}_{\mathbf i}(x) \mid x \in B(\infty)\}\subset \mathbb{Z}_{\ge 0}^{N}$$ 
is the set of integer points of a polyhedral cone called the \emph{string cone associated to $\mathbf i$}. 
\end{defi}

Following \cite{BZ} we introduce positive rational functions $\Psi^{\ii}_{\mathbf{j}} : (\mathbb{C}^*)^N \dashrightarrow (\mathbb{C}^*)^N$ such that the tropicalization $[\Psi^{\ii}_{\mathbf{j}}]_{trop}$ gives the transition map between the string cones associated to different reduced words $\ii, \mathbf{j}\in R(w_0)$.
\begin{defi} We define $\Psi^{\ii}_{\mathbf{j}} : (\mathbb{C}^*)^N \dashrightarrow (\mathbb{C}^*)^N$ as follows. If $\mathbf{j}\in R(w_0)$ is obtained from $\ii\in R(w_0)$ by a $3$-term braid move such that $(j_k,j_{k+1},j_{k+2})=(i_{k+1},i_k,i_{k+2})$, and $j_k=j_{k+2}$ and the entry $a_{i,j}$ of the Cartan matrix of $G$ is equal to $-1$. we set $y=\Psi^{\ii}_{\mathbf{j}}(x)$ with 
	$$
	\Psi^{\ii}_{\mathbf{j}}(x) = \left(x_1, \dots, x_{k-2}, \frac{x_{k}x_{k+1}}{x_{k-1}x_{k+1}+x_k}, x_{k-1}x_{k+1}, \frac{x_{k+1}x_{k-1}+x_k}{x_{k+1}}, x_{k+2}, \dots, x_N \right).
	$$
If $\mathbf{j}\in R(w_0)$ is obtained from $\ii\in R(w_0)$ by a $2$-term braid move such that $(j_k,j_{k+1})=(i_{k+1},i_k)$, we set $y=\Psi^{\ii}_{\mathbf{j}}(x)$ with 
	$$\Psi^{\ii}_{\mathbf{j}} \left(x_1, \dots, x_N \right) = \left( x_1, \dots, x_{k-1}, x_{k+1}, x_k, x_{k+2}, \dots, x_N\right).
	$$
	For arbitrary $\ii, \mathbf{j}\in R(w_0)$ we define $\Psi^{\ii}_{\mathbf{j}}:(\mathbb{C}^*)^N \rightarrow (\mathbb{C}^*)^N$ as the composition of the transition maps corresponding to a sequence of $2-$term and $3-$term braid moves transforming $\ii$ into $\mathbf{j}$.
\end{defi}
Recall that $[(\mathbb{C}^*)^N]_{trop}=\mathbb{Z}^{N}$. By \cite{Lit}, we have for every $x\in B(\infty)$
\begin{equation}\label{welldefi}
\text{str}_{\mathbf{j}}(x)=[\Psi^{\ii}_{\mathbf{j}}]_{trop}(\text{str}_{\ii}(x)).
\end{equation}

Defining inequalities of the string cones may also by obtained by the tropicalization of a positive function defined as follows.

\begin{defi}\label{varsigma} Let $\mathbf{i},\mathbf{j}\in R(w_0)$ and $i\in I$. We define $\varsigma_{\mathbf{i},i}$ to be the rational function on $(\mathbb{C}^*)^{N}$ uniquely determined by the following two conditions.
\begin{enumerate}
\item If $i_N=i$, then $\varsigma_{\mathbf{i},i}(x)={x_N}$.
\item We have
$\varsigma_{\mathbf{i},i}\circ \Psi^{\mathbf{j}}_{\mathbf{i}}=\varsigma_{\mathbf{j},i}.$
\end{enumerate}
\end{defi}

Note that the map $\varsigma_{\mathbf{i},i}$ is well-defined since $ \Psi^{\mathbf{j}}_{\mathbf{i}}$ does not depend on the chosen sequence of $2-$term and $3-$term braid moves transforming $\mathbf{i}$ into $\mathbf{j}$ by \eqref{welldefi}. In \cite{GKS20} we have shown that the tropicalization of the sum of $\varsigma_{\mathbf{i},i}$, $i\in I$, gives rise to the string cone inequalities:

\begin{prop}\cite[Proposition 3.5]{GKS20}\label{stringpos} For $\mathbf{i} \in R(w_0)$, we have
\begin{equation}
\mathcal{C}_{\mathbf{i}}=\{x \in \mathbb{R}^N \mid [\varsigma_{\mathbf{i},i}]_{trop}(x)\ge 0 \text{ for all }i\in I.\}.
\end{equation}
\end{prop}

\subsection{String cones from tropicalized $F-$polynomials}

In the following we recall from \cite{GKS20} how to recover the string cone inequalities from the Landau-Ginzburg potential $W$ on $\mathcal{X}$ with initial datum $\mathcal{X}_{\ii}$ given by $\Gamma_{\ii}$. 

\begin{defi}\label{dualCa} We define $\widehat{\text{CA}}_{\ii}\in \text{Hom}((\mathbb{C}^*)^{N}, {\mathcal{X}}_{{\ii}})$ as follows
$$\widehat{\text{CA}}_{\ii}(x)_j=\displaystyle\prod_{s\in [1,N]}X(\ii)_{s}^{{\{j,s\}}},$$
where $\{j,s\}:=\begin{cases} 1 & \text {if }j<s<j^+ \text{ and }i_j, i_{s} \text{ are connected in the Dynkin diagram of }G, \\ -1 & \text{if }s=j \text{ or }s=j^+, \\ 0 & \text{ else.} \end{cases}$
\end{defi}

\begin{prop}\label{iso} The map $\widehat{\text{CA}}_{\ii}\in \text{Hom} ((\mathbb{C}^*)^{N}, \mathcal{X}_{{\ii}})$ is an isomorphism of algebraic tori and satisfies
$\varsigma_{\mathbf{i},i}=\restr{W_i}{\mathcal{X}_{\ii}}\circ \widehat{\text{CA}}_{\ii}$.
\end{prop}

\begin{ex}\label{ex1} Let $G=\text{SL}_3(\mathbb{C})$ and $\ii=(1,2,1)\in R(w_0)$. We have
$$\widehat{\text{CA}}_{\ii}(x_1,x_2,x_3)=\left(\frac{x_2}{x_1x_3},\frac{x_3}{x_2},\frac{1}{x_3}\right).$$
Moreover,
$$\restr{W_2}{\mathcal{X}_{\ii}}=X^{-1}_2+X^{-1}_1X^{-1}_2, \quad \restr{W_1}{\mathcal{X}_{\ii}}=X^{-1}_3$$ and therefore
$$\varsigma_{\mathbf{i},2}=\restr{W_2}{\mathcal{X}_{\ii}}\circ \widehat{\text{CA}}_{\ii}=x_1+\frac{x_2}{x_3}, \quad \varsigma_{\mathbf{i},1}=\restr{W_1}{\mathcal{X}_{\ii}}\circ \widehat{\text{CA}}_{\ii}=x_3.$$
We recover the well-known string cone inequalities (\cite[Corollary 2]{Lit})
$$\mathcal{C}_{\mathbf{i}}=\{x \in \mathbb{R}^3 \mid x_1 \ge 0, \ x_2\ge x_3 \ge 0\}.$$
\end{ex}

The following theorem gives an expression of the defining inequalities of string cones as tropicalizations of dual $F$-polynomials without constant term of projective QP-representations.

Recall that the set of frozen vertices of $\Gamma_{\ii}$ is given by $\{\ell_i \mid i \in [1,n]\}$, where $\ell_i=\max \{1 \le j \le N \mid i_j=i\}$. Also recall from Section \ref{partcomp} that $(\Gamma_{\ii})^{\ell_i}$ is the full subgraph of $\Gamma_{\ii}$ spanned by all mutable vertices and $\ell_i$. Recall the non-degenerate potential $S_{\ii}$ on $\Gamma_{\ii}$ from Proposition \ref{nondegred} which restricts to a non-degenerate potential on $(\Gamma_{\ii})^{\ell_i}$, also called $S_{\ii}$ by slight abuse of notation.

\begin{thm}\label{stringasFpol} Let $\ii\in R(w_0)$. We have
$$\mathcal{C}_{\mathbf{i}}=\{x \in \mathbb{R}^N \mid [F^{\vee}_{{P}_{(Q_\ii, S_{\ii})}(\ell_i)}-1]_{trop}([\widehat{\text{CA}}_{\ii}]_{trop}(x)) \ge 0 \text{ for all } 1\le i \le n\},$$

where ${P}_{(Q_\ii, S_{\ii})}(\ell_i)$ is the indecomposable projective QP-representation of $(Q_{\ii}:=(\Gamma_{\ii})^{\ell_i},S_{\ii})$ to the vertex $\ell_i$.
\end{thm}

\begin{proof}
We have by Proposition \ref{stringpos} that
$$ \mathcal{C}_{\mathbf{i}}=\{x \in \mathbb{R}^N \mid [\varsigma_{\mathbf{i},i}]_{trop}(x)\ge 0 \text{ for all }i\in I.\}.$$
Let $x\in \mathbb{R}^N$. By Proposition \ref{iso}
$$[\varsigma_{\mathbf{i},i}]_{trop}(x)=[\restr{W_i}{\mathcal{X}_{\ii}}]_{trop}\circ [\widehat{\text{CA}}_{\ii}]_{trop}(x).$$
We have that $((\Gamma_{\ii})^{\ell_i},S(\ii))$ is non-degenerate by Proposition \ref{nondegred}, hence we may apply Theorem \ref{potasfpol} which yields
$$\restr{W_i}{\mathcal{X}_{\ii}}=F^{\vee}_{{P}_{(Q_\ii, S_{\ii})}(\ell_i)}-1$$
implying
$$[\restr{W_i}{\mathcal{X}_{\ii}}]_{trop}=[F^{\vee}_{{P}_{(Q_\ii, S_{\ii})}(\ell_i)}-1]_{trop}$$
which allows to conclude.
\end{proof}

\begin{ex} Let $G=\text{SL}_3(\mathbb{C})$ and $\ii=(1,2,1)\in R(w_0)$. The quiver $\Gamma_{\ii}$ looks as follows:

$$
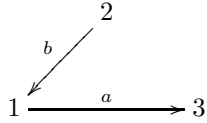
\xymatrix{
& 2 \ar[dl]_{b}  \\
1 \ar[rr]^{a} && 3 
}$$
with $\{\ell_1=2, \ell_2=3\}$ the set of frozen vertices and is thus the quiver from Example \ref{ex:SL3}. 

By Theorem \ref{stringasFpol}, we get
$$\mathcal{C}_{\mathbf{i}}=\{x\in \mathbb{R}^3 \mid -([\widehat{\text{CA}}_{\ii}]_{trop}(x))_3 \ge 0, \ -([\widehat{\text{CA}}_{\ii}]_{trop}(x))_2-([\widehat{\text{CA}}_{\ii}]_{trop}(x))_1 \ge 0, \ -([\widehat{\text{CA}}_{\ii}]_{trop}(x))_2 \ge 0 \}$$
and hence (see Example \ref{ex1})
$$\mathcal{C}_{\mathbf{i}}=\{x\in \mathbb{R}^3\mid x_3\ge 0, x_2\ge x_3, \ x_1\ge 0\}.$$
\end{ex}
Keeping the notation of Theorem \ref{stringasFpol}, in general the inequalities of $\mathcal{C}_{\mathbf{i}}$ obtained in this way may contain a lot of redundancies (see \cite{KS}). However, we get a non-redundant set of inequalities if we restrict ourselves to the vertices of the Newton polytope of $F^{\vee}_{{P}_{(Q_\ii, S_{\ii})}(\ell_i)}-1$. In the case that $F^{\vee}_{{P}_{(Q_\ii, S_{\ii})}(\ell_i)}-1$ is square-free, which means it has no multiple root over $\mathbb{C}$, the tropicalized summands of $F^{\vee}_{{P}_{(Q_\ii, S_{\ii})}(\ell_i)}-1$ yield an irredundant set of inequalities as explained in the following proposition. 

\begin{prop}
The map $[\widehat{\text{CA}}_{\ii}]_{trop}$ induces a one-to-one correspondence 
$$v\in \mathbb{Z}_{\ge 0}^N \mapsto \displaystyle\sum_{j=1}^{N}v_j(\widehat{\text{CA}}_{\ii}(x))_j \ge 0$$
between the the vertices of the Newton polytope of $F^{\vee}_{\mathcal{P}(\ell)}(X^{-1}(\ii),\ldots,X^{-1}(\ii))-1$ and facet defining inequalities of the string cone $\mathcal{C}_{\mathbf{i}}$ arising from the tropicalization of $\varsigma_{\mathbf{i},i}$.

Moreover, if $F^{\vee}_{\mathcal{P}_{((\Gamma_\ii)^{\ell_i}, S(\ii))}(\ell_i)}-1$ is square-free, then all coefficients of $F^{\vee}_{{P}_{(Q_\ii, S_{\ii})}(\ell_i)}$ are equal to $1$ and the vertices of the Newton polytope of $F^{\vee}_{{P}_{(Q_\ii, S_{\ii})}(\ell_i)}-1$ are given by the non-zero vectors $v=(v_j)\in \mathbb{Z}_{\ge 0}^N$ such that $\prod_{j=1}^{N}u_j^{v_j}$ is a monomial of $F^{\vee}_{{P}_{(Q_\ii, S_{\ii})}(\ell_i)}(u_1,\ldots,u_n)-1$.
\end{prop}

\begin{proof}
The first claim is a direct consequence of Theorem \ref{stringasFpol}. Note that if $F^{\vee}_{{P}_{(Q_\ii, S_{\ii})}(\ell_i)}(u_1,\ldots,u_N)-1$ is square-free, ${P}_{(Q_\ii, S_{\ii})}(\ell_i)$ is thin, i.e. its dimension vector consists only of entries less or equal to one. Therefore each non-trivial quiver Grassmannian $\Gr^{\ee}({P}_{(Q_\ii, S_{\ii})}(\ell_i))$ consists only of one point. We conclude that each coefficient is equal to $1$. 

The last claim is \cite[Proposition 6]{KS}.
\end{proof}

\subsection{String cones in type $A$}
Let $G=\text{SL}_n(\mathbb{C})$. In this case we have $W\cong S_n$, the symmetric group in $n$ letters.

Following \cite{GP}, we recall the notion of wiring diagrams, which are graphical presentations of a reduced word $\ii\in R(w_0)$.

\begin{defi}\label{defi:wire} 
	Let $\ii=(i_1,i_2,\ldots i_N)\in R(w_0)$. The \emph{wiring diagram} $\mathcal{D}_{\ii}$ consists of a family of $n$ piecewise straight lines, called \emph{wires}, which can be viewed as graphs of $n$ continuous piecewise-linear functions defined on the same interval. The wires $w_1,\ldots,w_{n}$ have labels in the set $[1,n]$. Each vertex of $\mathcal{D}_{\ii}$ (i.e. an intersection of two wires) represents a letter $j$ in $\ii$. If the vertex corresponds to the letter $j\in [1,n-1]$, then $j-1$ is equal to the number of wires running below this intersection. We call this number $j$ the \emph{level of the vertex} $v$. Denote the vertex $v$ by $[p,q]$ if it is the intersection of $w_p$ and $w_q$.

	The word $\ii$ can be read off from $\mathcal{D}_{\ii}$ by reading the levels of the vertices from left to right. 
\end{defi}

\begin{ex}\label{run} Let $n=5$ and $\ii=(2,1,2,3,4,3,2,1,3,2)$. The corresponding wiring diagram $\mathcal{D}_{\ii}$ is depicted below.

 
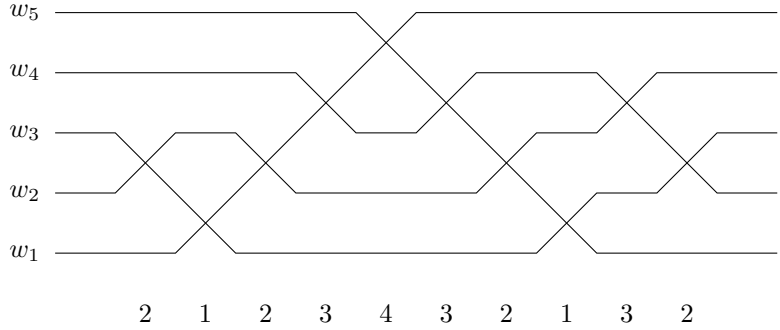
\begin{center}
	
	\begin{tikzpicture}[scale=.8]
	
	\node at (-.5,0) {$w_1$};
	\node at (-.5,1) {$w_2$};
	\node at (-.5,2) {$w_3$};
	\node at (-.5,3) {$w_4$};
	\node at (-.5,4) {$w_5$};
	\node at(1.5,-1){$2$};
	\node at(2.5,-1){$1$};
	\node at(3.5,-1){$2$};
	\node at(4.5,-1){$3$};
	\node at(5.5,-1){$4$};
	\node at(6.5,-1){$3$};
	\node at(7.5,-1){$2$};
	\node at(8.5,-1){$1$};
	\node at(9.5,-1){$3$};
	\node at(10.5,-1){$2$};
	
	\draw (0,0) --(1,0) --  (2,0) -- (3,1) --  (4,2) -- (5,3) --(6,4) -- (9,4) -- (12,4);
	\draw (0,1) -- (1,1) -- (2,2) -- (3,2) -- (4,1) -- (7,1) -- (8,2) -- (9,2) --(10,3) -- (12,3);
	\draw (0,2) --(1,2) -- (3,0) -- (8,0) -- (9,1) -- (10,1) -- (11,2) -- (12,2) ;
	\draw (0,3) --  (4,3) -- (5,2) -- (6,2) -- (7,3) --(9,3) -- (10,2) -- (11,1) -- (12,1) ;
	\draw (0,4) --  (5,4) -- (6,3) -- (7,2) -- (8,1) -- (9,0) -- (12,0) ;
	\end{tikzpicture}
\end{center}

\end{ex}
The condition $\ii\in R(w_0)$ implies that two lines $p,q$ with $p\ne q$ in $\mathcal{D}_{\ii}$ intersect exactly once. 
\begin{defi} Let $\ii\in R(w_0)$ and $\mathcal{D}_{\ii}$ be the corresponding wiring diagram. For $i\in [1,n-1]$ we denote by $\mathcal{D}_{\ii}(i)$ the oriented graph obtained from $\mathcal{D}_{\ii}$ by orienting its wires $p$ from left to right if $p\le i$, and from right to left if $p>i$. 
\end{defi}

\begin{ex} Let $i=4$ and $\mathcal{D}_{\ii}$ as in Example \ref{run}. The oriented graph $\mathcal{D}_{\ii}(4)$ looks as follows.

 
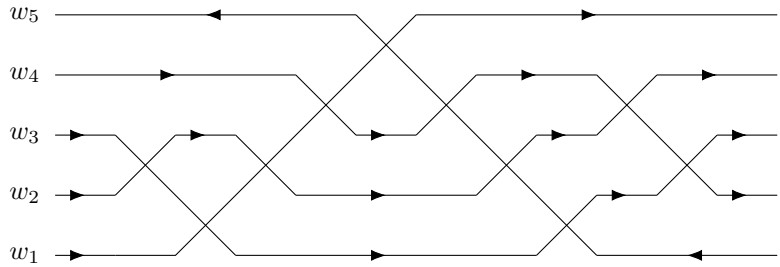
\begin{center}
	
	\begin{tikzpicture}[scale=.8]
		\begin{scope}[decoration={
		markings,
		mark=at position 0.5 with {\arrow{Latex[length=2mm]}}}
	] 
	\node at (-.5,0) {$w_1$};
\node at (-.5,1) {$w_2$};
\node at (-.5,2) {$w_3$};
\node at (-.5,3) {$w_4$};
\node at (-.5,4) {$w_5$};

\draw [postaction={decorate}] (0,0) -- (1,0);
\draw (1,0) --  (2,0) -- (3,1) --  (4,2) -- (5,3) --(6,4);
\draw [postaction={decorate}] (6,4) -- (12,4);
\draw [postaction={decorate}](0,1) -- (1,1);
\draw (1,1) -- (2,2);
\draw [postaction={decorate}] (2,2) -- (3,2);
\draw (3,2) -- (4,1);
\draw [postaction={decorate}] (4,1) -- (7,1);
\draw (7,1) -- (8,2);
\draw [postaction={decorate}] (8,2) -- (9,2);
\draw (9,2) --(10,3);
\draw [postaction={decorate}] (10,3) -- (12,3);
\draw [postaction={decorate}](0,2) --(1,2);
\draw (1,2) -- (3,0);
\draw [postaction={decorate}] (3,0) -- (8,0);
\draw (8,0) -- (9,1);
\draw [postaction={decorate}] (9,1) -- (10,1);
\draw (10,1) -- (11,2);
\draw [postaction={decorate}] (11,2) -- (12,2);
\draw [postaction={decorate}] (0,3) -- (4,3);
\draw (5,2) -- (4,3) ;
\draw [postaction={decorate}] (5,2) -- (6,2);
\draw (6,2) -- (7,3);
\draw [postaction={decorate}] (7,3) -- (9,3) ;
\draw (9,3) -- (10,2) -- (11,1);
\draw [postaction={decorate}] (11,1) -- (12,1);
\draw [postaction={decorate}] (5,4) -- (0,4);
\draw (5,4) -- (6,3) --  (7,2) -- (8,1) -- (9,0);
\draw [postaction={decorate}]  (12,0) -- (9,0) ;

	\end{scope}

	\end{tikzpicture}
\end{center}
\end{ex}

An oriented path in $\mathcal{D}_{\ii}(i)$ is a sequence $(v_1,\ldots,v_k)$ of vertices of $\mathcal{D}_{\ii}$ which are connected by oriented edges $v_1 \rightarrow v_2 \rightarrow \ldots \rightarrow v_{j}$ in $\mathcal{D}_{\ii}(i)$. 

\begin{defi} For $i\in [1,n-1]$ an \emph{$i$-crossing} is an oriented path $\gamma=(v_1,\ldots,v_j)$ in $\mathcal{D}_{\ii}(i)$ which starts with the rightmost vertex of the wire $i+1$ and ends with the rightmost vertex of the wire $i$. We say an $i$-crossing $\gamma$ is a rigorous path if $\gamma$ additionally satisfies the following condition: Whenever $v_j, v_{j+1}, v_{j+2}$ lie on the same wire $p$ in $\mathcal{D}_{\ii}$ and the vertex $v_{j+1}$ lies on the intersection the wires $p$ and $q$, then
	\begin{align*}
	p > q & \quad \text{if }q\le i, \\ 
	p<q & \quad \text{if }i  < p. \end{align*}
	In other words, the path $\gamma$ avoids the following two fragments.
	
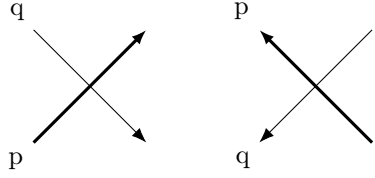
\begin{center}
		\begin{tikzpicture}[scale=.75]
		\draw[very thick, -{Latex[length=2mm]}] (0,0)node[below left]{p} -- (2,2);
		\draw[-{Latex[length=2mm]}] (0,2)node[above left]{q} -- (2,0);
		\draw[-{Latex[length=2mm]}] (6,2) -- (4,0) node[below left]{q};
		\draw[very thick, -{Latex[length=2mm]}] (6,0) -- (4,2)node[above left]{p};
		\end{tikzpicture}
	\end{center}
\end{defi}

\begin{ex}\label{rigpathex}
	Let $n=5$. The vertices lying on the bold path below form the rigorous path for $i=4$ $\gamma=([3,5],[2,5],[2,4],[3,4])$.

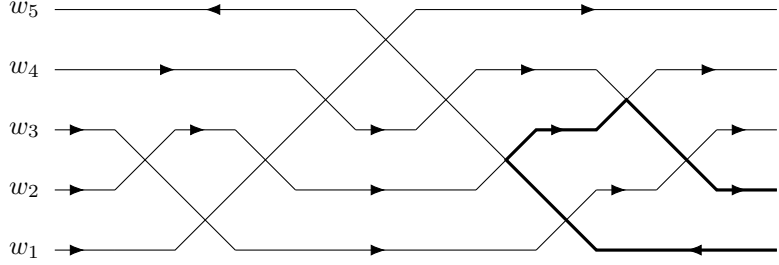
\begin{center}
	
	\begin{tikzpicture}[scale=.8]
\begin{scope}[decoration={
	markings,
	mark=at position 0.5 with {\arrow{Latex[length=2mm]}}}
] 
\node at (-.5,0) {$w_1$};
\node at (-.5,1) {$w_2$};
\node at (-.5,2) {$w_3$};
\node at (-.5,3) {$w_4$};
\node at (-.5,4) {$w_5$};

\draw [postaction={decorate}] (0,0) -- (1,0);
\draw (1,0) --  (2,0) -- (3,1) --  (4,2) -- (5,3) --(6,4);
\draw [postaction={decorate}] (6,4) -- (12,4);
\draw [postaction={decorate}](0,1) -- (1,1);
\draw (1,1) -- (2,2);
\draw [postaction={decorate}] (2,2) -- (3,2);
\draw (3,2) -- (4,1);
\draw [postaction={decorate}] (4,1) -- (7,1);
\draw (7,1) -- (8,2);
\draw [postaction={decorate}] (8,2) -- (9,2);
\draw (9,2) --(10,3);
\draw [postaction={decorate}] (10,3) -- (12,3);
\draw [postaction={decorate}](0,2) --(1,2);
\draw (1,2) -- (3,0);
\draw [postaction={decorate}] (3,0) -- (8,0);
\draw (8,0) -- (9,1);
\draw [postaction={decorate}] (9,1) -- (10,1);
\draw (10,1) -- (11,2);
\draw [postaction={decorate}] (11,2) -- (12,2);
\draw [postaction={decorate}] (0,3) -- (4,3);
\draw (5,2) -- (4,3) ;
\draw [postaction={decorate}] (5,2) -- (6,2);
\draw (6,2) -- (7,3);
\draw [postaction={decorate}] (7,3) -- (9,3) ;
\draw (9,3) -- (10,2) -- (11,1);
\draw [postaction={decorate}] (11,1) -- (12,1);
\draw [postaction={decorate}] (5,4) -- (0,4);
\draw (5,4) -- (6,3) --  (7,2) -- (8,1) -- (9,0);
\draw [postaction={decorate}]  (12,0) -- (9,0) ;
\draw[very thick] (12,0) -- (9,0) -- (8,1) -- (7.5,1.5) -- (8,2) -- (9,2) -- (9.5,2.5) -- (11,1) -- (12,1);

\end{scope}
\end{tikzpicture}
\end{center}

\end{ex}

\begin{defi} To each rigorous path $\gamma=(v_1,\ldots,v_j)$ we associate an integer vector $a_{\gamma}=(a_{[p,q]}^{(\gamma)})_{1\le p < q \le n}\in \mathbb{Z}^N$ (note that $N$ is equal to the number of vertices in $\mathcal{D}_{\ii}$) defined for $p< q$ as
$$a^{(\gamma)}_{[p,q]}=\begin{cases} 1 & \text{ if }\gamma \text{ travels from the wire }p \text{ to the wire }q \text{ at }[p,q], \\
-1 & \text{ if }\gamma \text{ travels from }q \text{ to }p \text{ at }[p,q], \\
0 & \text{else.}
\end{cases}$$
\end{defi}

\begin{ex} For $\gamma$ as in Example \ref{rigpathex}, we have $a^{(\gamma)}_{[2,3]}=0$, $a^{(\gamma)}_{[1,3]}=0$, $a^{(\gamma)}_{[1,2]}=0$, $a^{(\gamma)}_{[1,4]}=0$, $a^{(\gamma)}_{[1,5]}=0$, $a^{(\gamma)}_{[4,5]}=0$, $a^{(\gamma)}_{[2,5]}=-1$, $a^{(\gamma)}_{[3,5]}=0$, $a^{(\gamma)}_{[2,4]}=1$, $a^{(\gamma)}_{[3,4]}=0$.

\end{ex}

By \cite{GP}, we have

\begin{equation}\label{GPinqeu}
    \mathcal{C}_{\ii}=\{x \in \mathbb{R}^N \mid \displaystyle\sum_{j=1}^N a^{(\gamma)}_jx_j \ge 0 \text{ for all rigorous paths }\gamma \text{ in }\mathcal{D}_{\ii}\}.
\end{equation}

There is also an explicit description of $\restr{W}{\mathcal{X}_{\ii}}$ which we recall in the following. Note that the set of chambers of $\mathcal{D}_{\ii}$ except the unbounded chambers on the left hand side is equal to $N$ which gives another set of variables of cardinality $N$ by reading off that chambers from left to right. Here the $j^{th}$ chamber has the $j^{th}$ vertex $v_j$ of $\mathcal{D}_{\ii}$ as the leftmost vertex. Note that the map $\widehat{\text{CA}}_{\ii}$ from Definition \ref{dualCa} has a pictorial definition in terms of $\mathcal{D}_{\ii}$, namely we have

$$u_k=(\widehat{\text{CA}}_{\ii} x)_k=\displaystyle\prod_{j\in [1,N]}x_j^{\{k,j\}},$$
where 
$$\{k,j\}=\begin{cases} -1 & \text{ if }v_j \text{ is the right or leftmost vertex of the }k^{th} \text{ chamber}\\
1  & \text{ if }v_j \text{ is a local minimum or local maximum of the } k^{th} \text{ chamber with respect to the height.}\end{cases}$$

By \cite[Theorem 6.6]{GKS21} and independently by \cite{BF}, we have for every $i\in I$
\begin{equation}\label{typeAGHKKpot}\restr{W_i}{\mathcal{X}_{\ii}}=\displaystyle\sum_\gamma \prod_{\text{chambers }j \text{ enclosed by }\gamma}X_j^{-1},
\end{equation}
where the first sum varies over all rigorous paths in $\mathcal{D}_{\ii}(i)$. 

By the results of this paper we also get a description of the $QP-$quotient modules corresponding to each frozen vertex which is summarizes in the following theorem.
\begin{thm} The set of rigorous path in $\mathcal{D}_{\ii}(i)$ is in bijection with
\begin{itemize}
    \item[(i)] the inequalities of $\mathcal{C}_{\ii}$ arising from $[\varsigma_{\mathbf{i},i}]_{trop}$ defined in Definition \ref{varsigma} via the map
    $$\gamma \mapsto \displaystyle\sum_{j=1}^N a^{(\gamma)}_jx_j \ge 0$$
    \item[(ii)] the set of monomials of $\restr{W_i}{\mathcal{X}_{\ii}}$ via the map
    $$\gamma \mapsto \displaystyle\prod_{\text{chambers }j \text{ enclosed by }\gamma}X_j^{-1}$$,
    \item[(iii)] the isomorphism classes of quotient modules of the projective $(\Gamma_{\ii}^{\ell_i},S(\ii))$-representation ${P}_{(Q_\ii, S_{\ii})}(\ell_i)$  to the vertex $\ell_i$ via the map
    $$\gamma \mapsto \mathcal{M}(\gamma),$$
where $\mathcal{M}(\gamma)$ is a $(\Gamma_{\ii}^{\ell_i},S(\ii))$-representation with a $1$-dimensional vector space at each vertex which corresponds to a chamber enclosed by $\gamma$ and zero dimensional vector space at all other vertices.
\end{itemize}
\end{thm}
\begin{proof}
The bijection $(i)$ follows from \eqref{GPinqeu} together with the fact that this restricts to a bijection between the inequalities arizing from $[\varsigma_{\mathbf{i},i}]_{trop}$ and the rigorous paths in $\mathcal{D}_{\ii}(i)$ which was shown in \cite[Theorem 5.4, Remark 5.5]{GKS21}. Bijection $(ii)$ is \eqref{typeAGHKKpot}. 

By \eqref{typeAGHKKpot} and and Theorem \ref{potasfpol} it is immediate that the projective module ${P}_{(Q_{\ii},S_{\ii})}(\ell_i)$ is thin, i.e. its dimension vector only consists of $1$s and $0s$. Therefore all non-trivial quiver Grassmannians $Gr^{e}({P}_{(Q_{\ii},S_{\ii})}(\ell_i))$ consists of points and we conclude bijection $(iii)$.
\end{proof}

\end{document}